\documentclass[11pt]{amsart}

\usepackage{amssymb}
\usepackage{amsthm}
\usepackage{amsmath}
\usepackage{amsbsy}
\usepackage[all]{xy}
\usepackage{bm}
\usepackage{hyperref}
\usepackage{tikz}
\usepackage{array}
\usepackage{colortbl,xcolor}
\usepackage{ytableau}
\usepackage{hhline}
\usepackage{graphicx, wrapfig}
\usepackage[margin=1in]{geometry}

\allowdisplaybreaks

\newtheorem{theorem}{Theorem}[section]
\newtheorem*{theorem*}{Theorem}

\newtheorem*{proposition*}{Proposition}
\newtheorem{corollary}[theorem]{Corollary}
\newtheorem*{corollary*}{Corollary}

\newtheorem{conjecture}[theorem]{Conjecture}
\newtheorem{lemma}[theorem]{Lemma}

\newtheorem{definition}[theorem]{Definition}
\newtheorem*{definition*}{Definition}

\title{Ulam Sets in New Settings}

\author{Tej Bade}
\address{Phillips Exeter Academy, Exeter, NH 03833}
\email{tbade@exeter.edu}

\author{Kelly Cui}
\address{Leland High School, San Jose, CA 95120}
\email{kcui930@gmail.com}

\author{Antoine Labelle}
\address{Coll\`ege de Maisonneuve, Montreal, Quebec, Canada}
\email{e1995364@cmaisonneuve.qc.ca, antoinelab01@gmail.com}

\author{Deyuan Li}
\address{Yale University, New Haven, CT 06511}
\email{deyuan.li@yale.edu}

\begin{document}

\begin{abstract}
    The classical Ulam sequence is defined recursively as follows: $a_1=1$, $a_2=2$, and $a_n$, for $n > 2$, is the smallest integer not already in the sequence that can be written uniquely as the sum of two distinct earlier terms. This sequence is known for its mysterious quasi-periodic behavior and its surprising rigidity when we let $a_2$ vary. This definition can be generalized to other sets of generators in different settings with a binary operation and a valid notion of size. Since there is not always a natural linear ordering of the elements, the resulting collections are called Ulam sets. In this paper, we study Ulam sets in new settings. First, we investigate the structure of canonical Ulam sets in free groups; this is the first investigation of Ulam sets in noncommutative groups. We prove several symmetry results and prove a periodicity result for eventually periodic words with fixed prefixes. Then, we study Ulam sets in $\mathbb{Z}\times (\mathbb{Z}/n\mathbb{Z})$ and prove regularity for an infinite class of initial sets. We also examine an intriguing phenomenon about decompositions of later elements into sums of the generators. Finally, we consider $\mathcal{V}$-sets, a variant where the summands are not required to be distinct; we focus on $\mathcal{V}$-sets in $\mathbb{Z}^2$.
\end{abstract}
\maketitle
\section{Introduction}
\subsection{Background}
In 1964, Stanislaw Ulam \cite{ulam} introduced the following curious sequence of natural numbers: the first two elements are $1$ and $2$, and then we repeatedly choose the next element (in a greedy way) to be the smallest integer not already in the sequence that can be represented uniquely as the sum of two distinct previous elements. The first few terms are
$$1, 2, 3, 4, 6, 8, 11, 13, 16, 18, 26, \ldots.$$
This ``classical'' Ulam sequence is known to be quite chaotic, but it also exhibits very intriguing phenomena. Although the sequence appears to behave quite randomly, Steinerberger \cite{hidden signal} observed a ``hidden signal'': there seems to exist a real number $\lambda \approx 2.4434$ such that the Ulam sequence is far from uniformly distributed modulo $\lambda$. Since the publication of Steinerberger's result, the classical Ulam sequence has been the object of renewed attention \cite{algorithm, trillion, questions, ross-thesis}, but its behavior remains far from understood. Of course, one can obtain other Ulam sequences by starting with initial values other than $1,2$.  Some choices of initial values result in highly structured sequences \cite{regularity 4 v, finch regularity, regularity 2 v}, and others appear to result in highly irregular sequences like what appears in the classical Ulam sequence.

In 2018, Kravitz and Steinerberger \cite{noah} extended the notion of an Ulam sequence to settings other than the natural numbers, with the following caveat: when multiple elements have the same size, one is chosen arbitrarily to be added first. As a result, there can be no canonical notion of a sequence per se, and it makes more sense to study the unordered \emph{Ulam set}.  (The Ulam set arising from an initial set is well-defined as long as the notion of size satisfies some natural weak monotonicity conditions.)  Kravitz and Steinerberger focused on Ulam sets in $\mathbb{Z}^d$ (see also the work of \cite{rigidity2}), and they suggested several other settings for studying Ulam sets. The purpose of the present paper is to initiate the study of several of these variants.

One such variant, for example, consists of Ulam sets in $\mathbb{Z}\times (\mathbb{Z}/n\mathbb{Z})$. In this setting, Ulam sets can exhibit rather surprising new behavior. For instance, whereas Ulam sets in $\mathbb{Z}$ are always infinite, this is not always the case in $\mathbb{Z}\times (\mathbb{Z}/n\mathbb{Z})$. One example appears in $\mathbb{Z} \times (\mathbb{Z}/8\mathbb{Z})$, where the Ulam set generated by the initial set $\{(1,0),(1,1),(2,5)\}$ does not contain any points with $x$-coordinate larger than $51$. We will show in Section \ref{finiteness} that if $n<5$, then every Ulam set in $\mathbb{Z}\times (\mathbb{Z}/n\mathbb{Z})$ has infinitely many elements.

\subsection{General definition of Ulam sets and $\mathcal{V}$-sets}

Following the setup from Theorem 2 of \cite{noah}, we now formally define Ulam sets in general settings.  Let $G$ be a group (written multiplicatively), and fix a finite subset $S \subset G$ (the \emph{initial set} of generators).  Let $D_S$ denote the set of all elements of $G$ that can be expressed as a nonempty product of elements from $S$.  Suppose moreover that there exists a ``size'' function $f: D_S \to \mathbb{R}$ such that $f(xy)>\max\{f(x),f(y)\}$ for all $x, y \in D_S$, and $f^{-1}((-\infty,r])$ is finite for all $r \in \mathbb{R}$.  Then we define the \emph{Ulam set generated by S} (written $\mathcal{U}(S)$) as follows:
\begin{enumerate}
    \item We put the elements of $S$ into $\mathcal{U}(S)$.
    \item We then repeatedly add to $\mathcal{U}(S)$ the smallest (according to $f$) element of $D_S$ that is not already in $\mathcal{U}(S)$ and which is uniquely represented as the product of two distinct elements already in $\mathcal{U}(S)$. When there are multiple such smallest elements, we choose one to add arbitrarily.
\end{enumerate}
As described in \cite{noah}, since elements of $S$ affect the representations of only elements of greater size, valid elements of the same size can be added in $\mathcal{U}(S)$ in any arbitrary order without changing the resulting unordered set. When $G$ is abelian, we speak of sums rather than products and consider the representations $x+y$ and $y+x$ to be the same.

In \cite{V-sequences}, Kuca introduced a variant of Ulam sequences, called $\mathcal{V}$-sequences, in which the summands in a representation are not required to be distinct. Similarly, we can define \emph{$\mathcal{V}$-sets} in the same way as Ulam sets, except that the multiplicands or summands are not required to be distinct. When the ambient setting is clear, we denote by $\mathcal{V}(S)$ the $\mathcal{V}$-set generated by the initial set $S$.

\section{Main results}

\subsection{Ulam sets in free groups}

We begin in Section \ref{sec:free} with Ulam sets in the free group $F_2$.  We emphasize that this is the first ever investigation of Ulam sets in non-abelian groups.  In order to have a suitable notion of size, we will restrict our attention to the ``positive'' part of $F_2$, i.e., the set of nonidentity elements that can be expressed without the use of $0^{-1}$ or $1^{-1}$; this set can be identified with the set $\mathcal{W}_2$ of finite nonempty binary strings, where multiplication is given by concatenation. Given $w \in \mathcal{W}_2$, we define $f(w)$ to be the length of $w$.

We will focus on the case where the initial set $S$ contains only two elements; even this simple example exhibits substantially nontrivial phenomena.  In particular, we study $\mathcal{U}(\{0,1\})$, the Ulam set generated by the canonical elements $0$ and $1$.  The first few elements are $$\{0,1,01,10,001,011,100,110,\ldots \}.$$  Note that $\mathcal{U}(\{0,1\})$ exhibits universal behavior for the case where $S$ consists of two elements $v_1, v_2$ of the same size since the homomorphism sending $0$ to $v_1$ and $1$ to $v_2$ induces a bijection between the corresponding Ulam sets. We establish several symmetries of $\mathcal{U}(\{0,1\})$ and characterize the elements with exactly one $1$.

\newtheorem*{one1}{Theorem \ref{thm:one1}}
\begin{one1}
Let $u \in \mathcal{W}_2$ be a word of length $n$ with exactly one $1$, and let $i$ be the index of that $1$. The word $u$ is in $\mathcal{U}(\{0, 1\})$ if and only if $\binom{n-1}{i-1}$ is odd.
\end{one1}

We also work towards understanding the elements with exactly two $1$'s, and we show the somewhat surprising result that the gap between the $1$'s in such a word cannot be too large. We then exhibit infinite periodic structures in $\mathcal{U}(\{0,1\})$ in a way somewhat akin to the ``column phenomenon'' from \cite{noah}.

Finally, for the analogous $\mathcal{V}$-set (where we don't require the words forming a representation to be distinct), we find that whether or not a word is in $\mathcal{V}(\{0,1\})$ depends only on its length. The first few elements of this $\mathcal{V}$-set are $$\{0,1,00, 01,10, 11,0000,0001,\ldots \}.$$ We find a simple characterization of all the elements of $\mathcal{V}(\{0,1\})$:

\newtheorem*{freegroup-Vset}{Theorem \ref{thm:freegroup-Vset}}
\begin{freegroup-Vset}
A word $u$ of length $n$ is in $\mathcal{V}(\{0,1\})$ if and only if $n$ is a power of $2$.
\end{freegroup-Vset}

\subsection{Ulam sets in $\mathbb{Z} \times (\mathbb{Z}/n\mathbb{Z})$}

In Section \ref{ZxZn}, we investigate Ulam sets in $\mathbb{Z} \times (\mathbb{Z}/n\mathbb{Z})$, where the notion of size is $f((x,y))=x$ (and we choose $S$ so that each element has strictly positive first coordinate). As discussed in \cite{noah}, this setting is motivated by multiplicative Ulam sets in the complex numbers. Figure \ref{fig:ZxZ7}, for example, is the Ulam set $\mathcal{U}(\{(1,0), (1,1)\})$ in $\mathbb{Z}\times (\mathbb{Z}/7\mathbb{Z})$. We see that the set appears to exhibit chaotic behavior. It is not known whether this set is infinite or whether the density is the same in each row (when $y$ is fixed). Some other Ulam sets, however, can exhibit very structured properties.

\begin{figure}[htp!] 
    \centering
    \includegraphics[width=14cm]{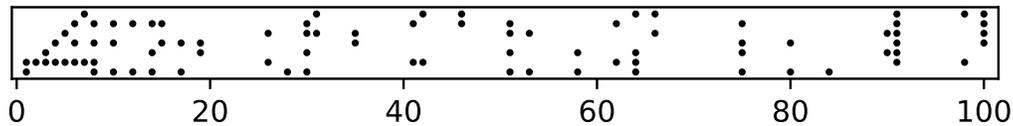}
    \caption{The Ulam set in $\mathbb{Z}\times (\mathbb{Z}/7\mathbb{Z})$ with initial set $\{(1,0), (1,1)\}$. Its associated lattice is generated by $(-7,7)$.}
    \label{fig:ZxZ7}
\end{figure}

We define the \emph{associated lattice} of the Ulam set generated by $\{v_1,\cdots, v_k\}$ to be the subgroup of $\mathbb{Z}^k$ consisting of the solutions to the equation $x_1v_1+\cdots+x_kv_k=0$. The structure of an Ulam set is uniquely determined by its associated lattice (see \cite{noah}), and it turns out that every Ulam set with two generators in an abelian group has the same associated lattice as that of an Ulam set in either $\mathbb{Z}^2$ or $\mathbb{Z} \times (\mathbb{Z}/n\mathbb{Z})$. Since the former setting has received attention previously in \cite{noah}, our investigation in this sense rounds out the study of Ulam sets with two generators in abelian groups.

Finch \cite{finch regularity} proved that an Ulam set in $\mathbb{Z}$ with finitely many even elements is regular (eventually periodic). Similarly, we establish a necessary and sufficient condition for an Ulam set in $\mathbb{Z} \times (\mathbb{Z}/n\mathbb{Z})$ to be regular.

\newtheorem*{E condition}{Theorem \ref{E condition}}
\begin{E condition}
An Ulam set $\mathcal{U}$ in $\mathbb{Z}\times (\mathbb{Z}/n\mathbb{Z})$ is regular if and only if there exists some regular subset $E\subset \mathbb{Z}\times (\mathbb{Z}/n\mathbb{Z})$ such that the sum of two elements of $E$ is never in $E$ and only finitely many elements of $\mathcal{U}$ are not in $E$.
\end{E condition}

In fact, Finch's characterization is a special case of our result, obtained by taking $n=1$ and $E$ to be the set of odd numbers.

Using these techniques, we show that Ulam sets with two generators and associated lattice generated by $(-2,b)$ for $b>3$ are regular, extending the work of Schmerl and Spiegel \cite{regularity 2 v}, who established the special case where $b$ is odd. These proofs are quite long and technical.

Since each element of an Ulam set (other than the elements of the generating set) has a unique representation as a sum of previous elements, it is possible to work backwards and keep track of overall decomposition into multiples of the original generators. We study these decompositions for Ulam sets in abelian groups with two generators and observe a very intriguing phenomenon: the ratio of the contributions of each generator is approximately the same for all elements. We show, as a weaker result, that the contributions of the two generators cannot be too skewed.

Finally, we investigate the conditions for an Ulam set in $\mathbb{Z} \times (\mathbb{Z}/n\mathbb{Z})$ to be finite.  (Recall that every Ulam set in $\mathbb{Z}^d$ is easily seen to be infinite.) As an example, the Ulam set $\mathcal{U}\{(1,0),(1,1),(2,5)\}$ in $\mathbb{Z} \times (\mathbb{Z}/8\mathbb{Z})$ is finite.

\begin{figure}[htp!] 
    \centering
    \includegraphics[width=14cm]{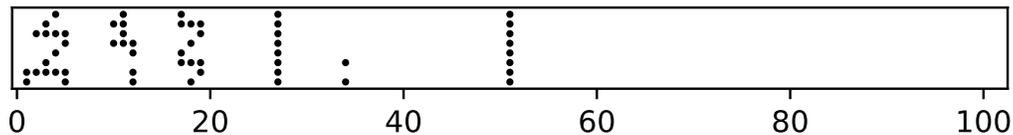}
    \caption{The Ulam set generated by $S=\{(1,0),(1,1),(2,5)\}$ in $\mathbb{Z} \times (\mathbb{Z}/8\mathbb{Z})$. The largest elements have $x$-coordinate $51$.}
    \label{fig:finite-set}
\end{figure}

We show that if such an Ulam set is finite, then it must have at least $5$ elements that assume the maximum value in the first coordinate.

\newtheorem*{col5}{Theorem \ref{5 or more in last column}}
\begin{col5}
Let $\mathcal{U}$ be a finite Ulam set in $\mathbb{Z} \times (\mathbb{Z}/n\mathbb{Z})$ and let $x_{\max}$ be the greatest $x$-coordinate of elements of $\mathcal{U}$. Then $\mathcal{U}$ contains at least $5$ elements of the form $(x_{\max},y)$.
\end{col5}

In particular, this implies that for $n<5$, all Ulam sets in $\mathbb{Z} \times (\mathbb{Z}/n\mathbb{Z})$ are infinite. Note that this is a tight bound since, for $n \ge 5$, the Ulam set $\mathcal{U}(\{(1,0),(1,1),\ldots, (1, n-2), (1, n-1)\})$ contains no element other than the generators.

\subsection{Higher-dimensional $\mathcal{V}$-sets}

In Section \ref{sec:V-sets}, we study $\mathcal{V}$-sets in $\mathbb{Z}^d$. Recall that a $\mathcal{V}$-set is a variant of an Ulam set where the summands in the representations need not be distinct. Kravitz and Steinerberger \cite{noah} demonstrated a ``column phenomenon'' for certain Ulam sets in $\mathbb{Z}^2$: if $S$ contains a single generator in the column $x=0$, then, for each fixed value of $x$, the set of values of $y$ such that $(x,y)$ is in $\mathcal{U}(S)$ is eventually periodic. An example of this behavior is shown in Figure~\ref{fig:colphen}.

\begin{figure}[htp!] 
    \centering
    \includegraphics[width=14cm]{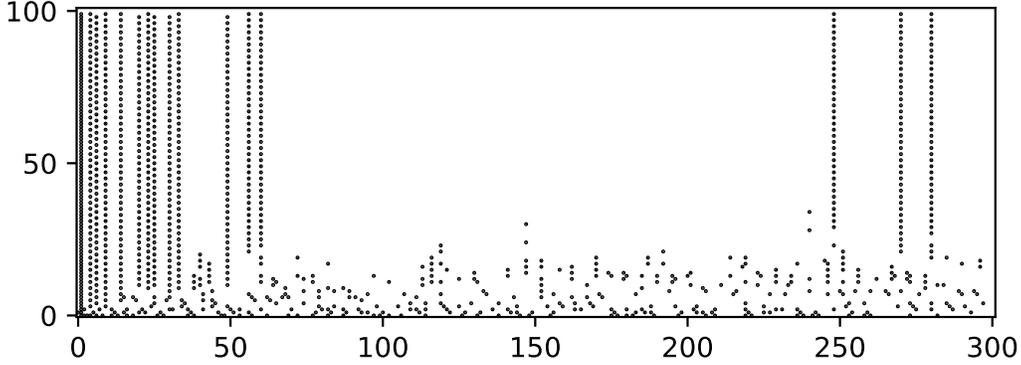}
    \caption{The Ulam set generated by $S=\{(1,0),(0,1),(2,0)\}$ in $\mathbb{Z}^2$. We see infinite periodic columns arising despite chaotic behavior near the $x$-axis.}
    \label{fig:colphen}
\end{figure}

We adapt their techniques to show that the same phenomenon persists for $\mathcal{V}$-sets in $\mathbb{Z}^2$.

\newtheorem*{colphen}{Theorem \ref{col phenomenon}}
\begin{colphen}
Let $\mathcal{S}$ either be an Ulam set or a $\mathcal{V}$-set in $\mathbb{Z}^2$ for which the column $x=0$ is eventually periodic. Then all of the columns of $\mathcal{S}$ are eventually periodic.
\end{colphen}

Next, we study the ``canonical'' $\mathcal{V}$-set in $\mathbb{Z}^2$ with the generating set $\{(0,1),(1,0)\}$. In contrast to the Ulam set setting, where starting with two generators in $\mathbb{Z}^2$ leads to simple lattice-like behavior, $\mathcal{V}(\{(0,1),(1,0)\})$ has a more complicated structure. We obtain an interesting result showing that it is not too chaotic.

\newtheorem*{v0cor}{Corollary \ref{V0 corollary}}
\begin{v0cor}
Let $$T=\{(1,1),(2,0),(0,2),(3,2),(2,3),(6,3),(3,6),(9,6),(6,9),(10,5),(5,10),(14,5),(5,14)\}.$$
Then every element of $\mathcal{V}(\{(0,1),(1,0)\})$ outside of $T$ (and the generators) must use an element of $T$ in its (unique) representation as a sum of two previous elements.
\end{v0cor}

Finally, in Section \ref{conclusion}, we raise several open questions and suggest avenues for future research on Ulam sets in settings other than $F_2$ and $\mathbb{Z}^d$.

\section{Ulam sets in Free Groups} \label{sec:free}

We begin by exhibiting a few elementary properties of $\mathcal{U}(\{0, 1\})$.

\subsection{Symmetries}

Let $u$ be a word on the alphabet $\{0,1\}$. We define the \emph{reverse} of $u$ (denoted $\overline{u}$) to be the word obtained by reversing the order of the characters in $u$.  We define the \emph{complement} of $u$ (denoted $\widehat{u}$) to be the word obtained by changing all of the $0$'s to $1$'s and $1$'s to $0$'s.
\begin{theorem}\label{thm:symmetry}
A word $u \in \mathcal{W}_2$ is contained in the Ulam set $\mathcal{U}(\{0, 1\})$ if and only if $\overline{u}$ is also contained in $\mathcal{U}(\{0, 1\})$. Similarly, $u \in \mathcal{U}(\{0, 1\})$ if and only if $\widehat{u} \in \mathcal{U}(\{0, 1\})$.
\end{theorem}
\begin{proof}
Because the two halves of the theorem are similar, we will be proving them in parallel. We proceed by induction on the length of $u$. The base case is where $u$ has length $1$.  We know that the words of length $1$ in $\mathcal{U}(\{0, 1\})$ are precisely $0$ and $1$; in particular, the reverse and the complement of each of these words are also in the set. Now we proceed with the inductive step. Assume that if a word of length $k < n$ is in the Ulam set $\mathcal{U}(\{0, 1\})$, then both its reverse and complement are also in $\mathcal{U}(\{0, 1\})$. Let $u$ be a word of length $n$ in $\mathcal{U}(\{0, 1\})$ with unique representation $u=v*w$ (where $v \neq w$). Note that $\overline{u} = \overline{w}*\overline{v}$ and $\widehat{u} = \widehat{v}*\widehat{w}$, where $\overline v$, $\overline w$, $\widehat v,$ and $\widehat{w}$ are in $\mathcal{U}(\{0, 1\})$ by the inductive hypothesis. This shows that $\overline u$ and $\widehat{u}$ each have at least one representation. 

For the first statement, assume for the sake of contradiction that there exists a second representation $\overline{u} = \overline{y}*\overline{x}$, where $\overline x$ and $\overline y$ are distinct elements of $\mathcal{U}(\{0, 1\})$ and, moreover, $\overline{x} \neq \overline{v}$. Since $\overline{u} = \overline{y}*\overline{x}$, we also have $u = x*y$, where $x, y \in \mathcal{U}(\{0, 1\})$ by the inductive hypothesis (because $x$ and $y$ are strictly shorter than $u$). However, this contradicts $u$ having a unique representation, so $\overline{u}$ must be in $\mathcal{U}(\{0, 1\})$. For the second statement, we assume for the sake of contradiction that there exists a second representation $\widehat{u} = \widehat{x}*\widehat{y}$, where $\widehat x, \widehat y \in \mathcal{U}(\{0, 1\})$ and $\widehat{x} \neq \widehat{v}$. Since $\widehat{u} = \widehat{x}*\widehat{y}$, we also have $u = x*y$, where $x,y \in \mathcal{U}(\{0, 1\})$ by the inductive hypothesis. Once again, this contradicts $u$ having a unique representation, so $\widehat{u} \in \mathcal{U}(\{0, 1\})$.
\end{proof}

We define a \emph{palindrome} to be a word $p$ such that $p = \overline{p}$.

\begin{corollary}
The only palindromes of odd length in $\mathcal{U}(\{0, 1\})$ are $0$ and $1$. 
\end{corollary}

\begin{proof}
Let $p$ be a palindrome with odd length other than $0$ and $1$. In particular, $p$ has length at least $3$. We will show that the existence of one representation of $p$ as the concatenation of two distinct previous elements of $\mathcal{U}(\{0, 1\})$ implies the existence of another such representation of $p$. Suppose $p = u*v$, where $u, v \in \mathcal{U}(\{0, 1\})$. By Theorem \ref{thm:symmetry}, the elements $\overline{u}$ and $\overline{v}$ must also be in $\mathcal{U}(\{0, 1\})$. Since $p$ is a palindrome, it may also be represented as $p = \overline{p}=\overline{v}*\overline{u}$, where $\overline{v} \neq u$ since $u$ and $v$ must have different lengths as $p$ has odd length. Therefore, $p$ cannot be in $\mathcal{U}(\{0, 1\})$.
\end{proof}

\subsection{Characterizing words with a small number of $1$'s}

We have demonstrated symmetries among words in $\mathcal{U}(\{0, 1\}).$ Now we focus on figuring out whether a specific word is in the set based on its actual sequence of $0$'s and $1$'s. Specifically, we begin with words with a small number of $1$'s.  Because of Theorem~\ref{thm:symmetry}, this discussion also pertains to words with a small number of $0$'s. We define the \emph{index} of a letter in a word to be the number of letters to its left (including the letter itself). 
\begin{theorem}\label{thm:one1}
Let $u \in \mathcal{W}_2$ be a word of length $n$ with exactly one $1$, and let $i$ be the index of that $1$. The word $u$ is in $\mathcal{U}(\{0, 1\})$ if and only if $\binom{n-1}{i-1}$ is odd.
\end{theorem}
\begin{proof}
Once again, we proceed with induction on the length of the word. The base case is where $u$ has length $1$ or $2$. The only word of length $1$ in $\mathcal{W}_2$ with exactly one $1$ is $1$. In this case, $n = 1$, $i = 1,$ $\binom{0}{0} = 1$ is odd, and the element $1$ is in fact in $\mathcal{U}(\{0, 1\})$. The words of length $2$ in $\mathcal{W}_2$ with exactly one $1$ are $01$ and $10$. For the former, $n = 2$, $i = 2$, $\binom{1}{1} = 1$ is odd, and $01$ is in $\mathcal{U}(\{0, 1\}).$ For the latter, $n = 2$, $i = 1$, $\binom{1}{0} = 1$ is odd, and $10$ is also in $\mathcal{U}(\{0, 1\})$. 

For the inductive step, we assume that the statement holds for all words of length $k \le n$. Let $u \in \mathcal{W}_2$ be a word of length $n + 1$ with exactly one $1$, and let $i$ be the index of the $1$. Note that $0$ is the only element of $\mathcal{U}(\{0, 1\})$ consisting of all $0$'s.  Thus, the only way to obtain $u$ as a concatenation of two previous elements is to concatenate $0$ to a word of length $n$ with exactly one $1$. Let $v$ denote the word of length $n$ with exactly one $1$ at index $i$ and $0$'s everywhere else; let $w$ denote the word of length $n$ with exactly one $1$ at index $i-1$ and $0$'s everywhere else. We see that $u \in \mathcal{U}(\{0, 1\})$ if and only if exactly one of $v$ and $w$ is in $\mathcal{U}(\{0, 1\})$. The word $v$ has binomial coefficient $\binom{n-1}{i - 1}$ and the word $w$ has binomial coefficient $\binom{n - 1}{i - 2}$.

By our inductive hypothesis, $u$ has a unique representation and is in $\mathcal{U}(\{0, 1\})$ if and only if exactly one of $\binom{n - 1}{i - 2}$ and $\binom{n - 1}{i - 1}$ is odd. Therefore, we must have 
$$\binom{n - 1}{i - 2} + \binom{n - 1}{i - 1} \equiv 1 \pmod{2}.$$
By Pascal's Identity, $\binom{n - 1}{i - 2} + \binom{n - 1}{i - 1} = \binom{n}{i - 1}$, so our equation becomes $\binom{n}{i - 1} \equiv 1 \pmod{2}$. Note that $\binom{n}{i - 1}$ is the binomial coefficient for $u$, so $u \in \mathcal{U}(\{0, 1\})$ if and only if its binomial coefficient is odd, which completes our induction.
\end{proof}
We can now tell if a word with exactly one $1$ is in $\mathcal{U}(\{0, 1\})$ by examining the parity of its corresponding binomial coefficient. Note that a word $u \in \mathcal{U}(\{0, 1\})$ of length $n$ with exactly one $1$ at index $i$ has the binomial coefficient $\binom{n - 1}{i - 1}$, which corresponds to the $i$-th number of the $(n-1)$-th row of Pascal's Triangle.

Gould's sequence \cite{gould} is an integer sequence that counts the number of odd terms in each row of Pascal's Triangle. Specifically, the $n$-th term of Gould's sequence is the number of odd numbers in the $(n-1)$-th row of Pascal's Triangle. The first few terms of Gould's sequence are $1, 2, 2, 4, 2, 4, 4, 8, 2, 4, 4, 8, 4, 8, 8, 16$.

\begin{corollary}
The number of words of length $n$ in $\mathcal{U}(\{0, 1\})$ with exactly one $1$ is the $n$-th number in Gould's sequence.
\end{corollary}

Theorem \ref{thm:one1} fully characterizes all words in $\mathcal{U}(\{0,1\})$ with exactly one $1$. We now analyze words in $\mathcal{U}(\{0,1\})$ with exactly two $1$'s.

\begin{theorem} \label{thm:consec1}
Let $u \in \mathcal{W}_2$ be a word of length $n \ge 2$ with exactly two $1$'s such that the $1$'s are consecutive. The word $u$ is in $\mathcal{U}(\{0, 1\})$ if and only if $n$ is odd.
\end{theorem}
\begin{proof}
We proceed by induction on the length of the word. The base cases are $n=2$ and $n=3$. The only word of length $2$ with two consecutive $1$'s (and the rest $0$'s) is $11$; this is not in $\mathcal{U}(\{0, 1\})$. The words of length $3$ with two consecutive $1$'s (and the rest $0$'s) are $011$ and $110$; both are in $\mathcal{U}(\{0, 1\})$ by direct computation. 

We now perform the inductive step.  Assume that the theorem holds for all words of length strictly smaller than $n$.  Consider the word $u$ of length $n$ which consists of $k$ $0$'s, followed by two $1$'s and then $\ell$ more $0$'s, where we must have $k \ge 1$ or $\ell \ge 1$ since $n \ge 4$. By Theorem \ref{thm:symmetry}, we can assume without loss of generality that $k \ge 1$. One representation of $u$ is the concatenation
$$u=\underbrace{0 \cdots 0}_{k}11\underbrace{0 \cdots 0}_{\ell}=\underbrace{0 \cdots 0}_{k}1*1\underbrace{0 \cdots 0}_{\ell},$$
where the elements $\underbrace{0 \cdots 0}_{k}1$ and $1\underbrace{0 \cdots 0}_{\ell}$ are in $\mathcal{U}(\{0, 1\})$ by Theorem \ref{thm:one1}.
If $n$ is even, a second representation is given by 
$$u=\underbrace{0 \cdots 0}_{k}11\underbrace{0 \cdots 0}_{\ell}=0*\underbrace{0 \cdots 0}_{k-1}11\underbrace{0 \cdots 0}_{\ell},$$
meaning that $u$ is not in $\mathcal{U}(\{0, 1\})$. If $n$ is odd, we claim that there is no second representation. Recall that $0$ is the only element of $\mathcal{U}(\{0, 1\})$ consisting of all $0$'s, so the only possible second representation of $u$ would have the form 
$$u=0*\underbrace{0 \cdots 0}_{k-1}11\underbrace{0 \cdots 0}_{\ell} \quad \text{or} \quad u=\underbrace{0 \cdots 0}_{k}11\underbrace{0 \cdots 0}_{\ell-1}*0.$$
But, by the inductive hypothesis, neither $\underbrace{0 \cdots 0}_{k-1}11\underbrace{0 \cdots 0}_{\ell}$ nor $\underbrace{0 \cdots 0}_{k}11\underbrace{0 \cdots 0}_{\ell-1}$ is in $\mathcal{U}(\{0, 1\})$, which completes the proof.
\end{proof}

\begin{theorem} \label{thm:101}
Let $u \in \mathcal{W}_2$ be a word of length $n \ge 5$ with exactly two $1$'s such that the $1$'s are separated by exactly one $0$. The word $u$ is in $\mathcal{U}(\{0, 1\})$ if and only if $n$ is even.
\end{theorem}
\begin{proof}
Again, we proceed by induction on $n$. The base cases $n=5$ and $n=6$ are easy to check by direct computation. 

For the inductive step, assume that the statement holds for all values strictly less than $n$. Consider a word $u$ of length $n$ consisting of $k$ $0$'s followed by $101$ and then $\ell$ $0$'s. Since $n \ge 5$, we must have $k \ge 1$ or $\ell \ge 1$. By Theorem \ref{thm:symmetry}, we can assume without loss of generality that $k \ge 1$. 

\textbf{Case $1$: $\ell=0$.} One representation of $u$ is the concatenation
$$u=\underbrace{0 \cdots 0}_{k}101=\underbrace{0 \cdots 0}_{k}1*01,$$
where $01$ and $\underbrace{0 \cdots 0}_{k}1$ are both elements of $\mathcal{U}(\{0,1\})$ by Theorem \ref{thm:one1}. If $n$ is odd, then a second representation of $u$ as the concatenation of words in $\mathcal{U}(\{0,1\})$ is given by 
$$u=\underbrace{0 \cdots 0}_{k}101=\underbrace{0 \cdots 0}_{k}10*1,$$
meaning that $u$ is not in $\mathcal{U}(\{0,1\})$. If $n$ is even, this second representation fails, since $\underbrace{0 \cdots 0}_{k}10$ is not an element of $\mathcal{U}(\{0,1\})$ by Theorem \ref{thm:one1}. We claim there is no second representation if $n$ is even. Since $0$ is the only element of $\mathcal{U}(\{0,1\})$, the only other representation of $u$ would be
$$u=0*\underbrace{0 \cdots 0}_{k-1}101.$$
But by the inductive hypothesis, $\underbrace{0 \cdots 0}_{k-1}101$ is not in $\mathcal{U}(\{0,1\})$, so $u$ must be in $\mathcal{U}(\{0,1\})$.

\textbf{Case 2: $\ell \ge 1$.} If $n$ is odd, then $u$ can be represented as 
$$u=\underbrace{0 \cdots 0}_{k}101\underbrace{0 \cdots 0}_{\ell}=0*\underbrace{0 \cdots 0}_{k-1}101\underbrace{0 \cdots 0}_{\ell}=\underbrace{0 \cdots 0}_{k}101\underbrace{0 \cdots 0}_{\ell-1}*0,$$
where $\underbrace{0 \cdots 0}_{k-1}101\underbrace{0 \cdots 0}_{\ell}$ and $\underbrace{0 \cdots 0}_{k}101\underbrace{0 \cdots 0}_{\ell-1}$ are both in $\mathcal{U}(\{0,1\})$ by our inductive hypothesis. Thus, $u$ is excluded from $\mathcal{U}(\{0,1\})$ when $n$ is odd.

If $n$ is even, exactly one of $\underbrace{0 \cdots 0}_{k}10$ and $01\underbrace{0 \cdots 0}_{\ell}$ will have even length. By Theorem \ref{thm:one1}, exactly one of
$$u=\underbrace{0 \cdots 0}_{k}101\underbrace{0 \cdots 0}_{\ell}=\underbrace{0 \cdots 0}_{k}10*1\underbrace{0 \cdots 0}_{\ell} \quad \text{and} \quad u=\underbrace{0 \cdots 0}_{k}1*01\underbrace{0 \cdots 0}_{\ell}$$
will be a representation of $u$ as the concatenation of words in $\mathcal{U}(\{0,1\})$.

Since $0$ is the only word in $\mathcal{U}(\{0,1\})$ that consists entirely of $0$'s, the remaining two possible representations of $u$ are
$$u=\underbrace{0 \cdots 0}_{k}101\underbrace{0 \cdots 0}_{\ell}=0*\underbrace{0 \cdots 0}_{k-1}101\underbrace{0 \cdots 0}_{\ell} \quad \text{and} \quad u=\underbrace{0 \cdots 0}_{k}101\underbrace{0 \cdots 0}_{\ell-1}*0.$$

But by our inductive hypothesis, neither $\underbrace{0 \cdots 0}_{k-1}101\underbrace{0 \cdots 0}_{\ell}$ nor $\underbrace{0 \cdots 0}_{k}101\underbrace{0 \cdots 0}_{\ell-1}$ is in $\mathcal{U}(\{0,1\}),$ which means $u$ is in $\mathcal{U}(\{0,1\})$ if $n$ is even.
\end{proof}

While Theorem \ref{thm:consec1} and Theorem \ref{thm:101} characterize an infinite set of words with exactly two $1$'s in $\mathcal{U}(\{0,1\})$ satisfying specific conditions, they do not apply to all words with exactly two $1$'s. We now give a general necessary condition for a word with two $1$'s to be in $\mathcal{U}(\{0,1\})$.

\begin{theorem} \label{thm:dist1}
Let $u \in \mathcal{W}_2$ be a word of length $n \ge 2$ with exactly two $1$'s, and let $i_1 < i_2$ be the indices of the two $1$'s. If $u$ is in $\mathcal{U}(\{0, 1\})$, then 
$$i_2-i_1 \le \frac{n}{2}.$$
\end{theorem}

We first require a preparatory lemma.  For nonnegative integers $a$, $b$, and $n$, let $$S_{a,b,n}=\left\{ i : \binom{i}{a}\binom{n-i}{b} \equiv 1 \pmod{2} \right\}.$$ $S_{a,b,n}$ roughly counts the ways in which a string of length $n+2$ with two $1$'s at the indices $a+1$ and $n+2-b$ can be formed as the concatenation of two strings in $\mathcal{U}(\{0,1\})$, each containing one $1$. Thus, we must understand $S_{a,b,n}$ because it yields insight into different representations of words in $\mathcal{W}_2$ with exactly two $1$'s.
\begin{lemma} \label{lem: sizeS}
For all nonnegative integers $a,b, n$ such that $a+b < \frac{n}{2},$ we have $|S_{a,b,n}| \neq 1.$
\end{lemma}
\begin{proof}
Assume that $|S_{a,b,n}| \ge 1$, so that there is some $i \in S_{a,b,n}$.  We have

$$\binom{i}{a}\binom{n-i}{b}\equiv 1 \pmod{2}.$$

Given $i \in S_{a,b,n}$, we aim to exhibit a second element $i' \in S_{a,b,n}$.

Consider the binary representations $a=a_k2^k+a_{k-1}2^{k-1}+\cdots+a_12+a_0=\overline{a_ka_{k-1}\cdots a_1a_0}$, $b=\overline{b_kb_{k-1}\cdots b_1b_0}$, $i=\overline{i_ki_{k-1}\cdots i_1i_0}$, and $n-i=\overline{(n-i)_k(n-i)_{k-1}\cdots (n-i)_1(n-i)_0}$, where $k$ is sufficiently large so that $a$, $b$, $i$, and $n-i$ all have leading $0$'s. Let $j$ with $0 \le j \le k$ be the maximum value such that $a_j=b_j=0$ but at least one of $i_j=1$ or $(n-i)_j=1$; such a $j$ exists since $a+b < \frac{n}{2}$. If exactly one of $i_j$ and $(n-i)_j$ equals $1$, then assume without loss of generality that $i_j=1$ and $(n-i)_j=0$. This loses no generality because we can swap $a$ with $b$ and $i$ with $n-i$. By Lucas' Theorem \cite{lucas theorem}, the choice $i'=i-2^j$ and $n-i'=(n-i)+2^j$ satisfies
$$\binom{i'}{a}\binom{n-i'}{b} \equiv 1 \pmod{2}.$$

On the other hand, if $i_j=(n-i)_j=1$, then let $m$ with $k \ge m >j$ be the smallest value such that $i_m=0$ or $(n-i)_m=0$. Without loss of generality, assume that $i_m=0$. 

For $m > x > j$, at least one of $a_{x}=0$ or $b_{x}=0$; otherwise, the maximum possible value of $n$ would be $2(a-2^x)+2(b-2^x)+2(1+2+\cdots+2^{x-1}+2^x)$, since $j$ is the largest index where $a_j=b_j=0$ but $i_j=1$ or $(n-i)_j=1$. But then $\frac{n}{2} \le a-2^x+b+1+2+\cdots+2^{x-1} < a+b$. Thus, if $a_x=1$, then $b_x=0$. Furthermore, $m$ is defined such that $i_x=(n-i)_x=1$ for all $m>x>j$.

By Lucas' Theorem \cite{lucas theorem}, 
$$i'=i-2^j-\left (\sum_{j<x<m, a_x=0}2^x\right )+2^m$$ with 
$$n-i'=(n-i)-2^j-\left( \sum_{j<x<m, a_x=1}2^x\right )$$ therefore also satisfies
$$\binom{i'}{a}\binom{n-i'}{b} \equiv 1 \pmod{2},$$
which implies that $|S_{a,b,n}| \neq 1$.
\end{proof}

Now we are ready to prove Theorem \ref{thm:dist1}.
\begin{proof}[Proof of Theorem \ref{thm:dist1}]
For any word $u$ of length $n$ with exactly two $1$'s, at indices $i_1$ and $i_2$, that satisfies $i_2-i_1 > \frac{n}{2}$, we wish to prove $u$ is not in $\mathcal{U}(\{0, 1\})$.

In particular, we wish to prove that $u$ can be constructed as the concatenation of two smaller words in $\mathcal{U}(\{0, 1\})$ in either zero or at least two ways.  We consider the possible representations of $u$ as a concatenation of two previous elements, $u_1$ and $u_2$, of $\mathcal{U}(\{0, 1\})$. First, we rule out the case where one of $u_1$ and $u_2$ contains both $1$'s, by induction on $n$. Thus, we restrict our attention to expressions of $u=u_1*u_2$ as the concatenation of two smaller words $u_1$ and $u_2$ in $\mathcal{U}(\{0, 1\})$, each with exactly one $1$. 

We note that the index of the $1$ in $u_1$ is $i_1$, while the index of the $1$, from right to left, in $u_2$ is $n-i_2+1$. By Theorem \ref{thm:one1}, we therefore wish to prove that if $k$ is the length of $u_1$ (so that $u_2$ has length $n-k$), then there exists either $0$ or at least $2$ possible values of $k$ such that 
$$\binom{k-1}{i_1-1} \equiv 1 \pmod{2} \quad \text{and} \quad \binom{(n-2)-(k-1)}{n-i_2} \equiv 1 \pmod{2}.$$ Note that since $i_2-i_1 > \frac{n}{2}$, $(i_1-1)+(n-i_2) < \frac{n-2}{2}$. We therefore see that the number of ways to express $u$ as the concatenation of two words in $\mathcal{U}(\{0,1\})$, each consisting of one $1$, is $|S_{i_1-1, n-i_2, n-2}|$. By Lemma \ref{lem: sizeS}, however, $|S_{i_1-1, n-i_2, n-2}|\neq 1$; hence, all such $u$ with $i_2-i_1 > \frac{n}{2}$ are excluded from $\mathcal{U}(\{0,1\})$.
\end{proof}
We note that $i_2-i_1 \le \frac{n}{2}$ is a tight bound and that Lemma \ref{lem: sizeS} no longer holds when $a+b = \frac{n}{2}$. In particular, $001001$, $000001000001$, $000010000010$, and $00001000000100$ are all examples of words in $\mathcal{U}(\{0,1\})$ that satisfy $i_2-i_1 = \frac{n}{2}$.

\subsection{The column phenomenon} \label{free-col}
Kravitz and Steinerberger observed a column phenomenon in Ulam sets of certain commutative settings \cite{noah}. We extend this notion to the noncommutative setting of free groups and prove a similar result for $\mathcal{U}(\{0,1\})$.

They proved that in certain commutative Ulam sets in $\mathbb{Z}^2$, all the columns (each obtained by fixing a $x$-coordinate) are eventually periodic. This does not readily apply to $\mathcal{U}(\{0,1\})$; instead, we extend this notion of columns by fixing a suitable infinitely long word and considering all its prefixes.

More precisely, for a word $t \in \mathcal{W}_2$ of length $m$, define $T=t*t*t*\cdots$ to be the infinitely long word consisting of the concatenation of infinitely many copies of $t$.

For $0\le i <m$, denote by $T_{i,j}$ the subword of $T$ consisting of the $j$ consecutive letters in $T$ starting after the $i$-th index. For simplicity, let $T_k$ be another expression for $T_{0,k}$.

\begin{theorem} \label{thm: free-periodicity}
If for any $0 \le i <m$, the set of values $k$ for which $T_{i,k}$ is in $\mathcal{U}(\{0,1\})$ is eventually periodic, then for all words $u \in \mathcal{W}_2$, the set of values $k$ for which $u*T_k$ is in $\mathcal{U}(\{0,1\})$ is also eventually periodic.
\end{theorem}
\begin{proof}
We proceed by induction on the length, $n$, of word $u$. For the base case, $n=0$, we already assumed that the values $k$ for which $T_k$ is in $\mathcal{U}(\{0,1\})$ are eventually periodic.

For the inductive step, assume that the theorem holds for all words of length less than $n$. Then for a word $u$ of length $n$, there are two potential ways to represent $u*T_k$ as the product of two smaller terms in $\mathcal{U}(\{0,1\})$:
\begin{enumerate}
     \item $u*T_k=u_1*(u_2*T_k)$, where $u_1*u_2=u$ and $u_1$ and $u_2$ are words of positive length.
     \item $u*T_k=(u*T_{k-k'})*T_{i,k'}$, where $i \equiv k-k' \pmod{m}$, $0 \le i < m$, and we consider $u*T_0$ to be equal to $u$.
\end{enumerate}

Let's first disregard the second case. Note that there are $n-1$ ways to represent $u$ as $u_1*u_2$, as necessary for the first case. By our inductive hypothesis, the set of values $k$ such that $u_2*T_k$ is in $\mathcal{U}(\{0,1\})$ is eventually periodic for all $n-1$ such $u_2$.  Thus, if $P$ is the least common multiple of their periods, there exists some large $K$ such that, for any $u_2$ and for $k>K$, $u_2*T_k\in \mathcal{U}(\{0,1\})$ if and only if $u_2*T_{k+P} \in \mathcal{U}(\{0,1\})$. Note that the number of representations of $u*T_k$ as $u_1*(u_2*T_k)$ depends solely on whether each $u_2*T_k$ is in $\mathcal{U}(\{0,1\})$, since each $u_1$ is fixed. But for $k>K$, $u_2*T_k$ is in $\mathcal{U}(\{0,1\})$ if and only if $u_2*T_{k+P}$ is included. Thus, if we define
\[ 
b_{u,k}= \begin{cases} 
      0 & \text{if there are no such representations of $u*T_k$}\\
      1 & \text{if there is exactly one such representation of $u*T_k$} \\
      2_+ & \text{otherwise}
      \end{cases},
\]
then $(b_{u,k})$ is eventually periodic in $k$ with period $P$.

Now, we account for the second case. We will show that the periodicity still holds (though the period can change) when also considering representations of the form $u*T_k=(u*T_{k-k'})*T_{0, k'}$ (when $i=0$ and so $k \equiv k' \equiv 0 \pmod{m}$). We assumed that the values $k'$ such that $T_{k'}$ is in $\mathcal{U}(\{0,1\})$ are eventually periodic. We can therefore split up the values $k'$ for which $T_{k'}$ is in $\mathcal{U}(\{0,1\})$ into a finite non-periodic transient phase and a periodic phase. 

Similar to $(b_{u,k})$, let $(c_{u,k})$ denote the number of representations of $u*T_k$ where we now also consider possible representations $u*T_k=(u*T_{k-k'})*T_{0, k'}$ where $k'$ is in the periodic phase. For simplicity, we similarly let $c_{u,k}$ to be equal to $2_+$ if it has at least $2$ such representations.

Let $P_0$ be the period of the periodic phase for $T_{0,k'}$, and let $C$ be the set of all congruence classes modulo $PP_0$ that contain infinitely many values $k'$ with $T_{0,k'} \in \mathcal{U}(\{0,1\})$. For a fixed equivalence class, $R$, modulo $PP_0$, consider all equivalence classes $R-S$, where $S \in C$. Then any $k-k' \equiv 0 \pmod{m}$ in any of the equivalence classes $R-S$ such that $u*T_{k-k'} \in \mathcal{U}(\{0,1\})$ yields a representation of $u*T_k$ for sufficiently large $k \in R$. Thus, if there are two or more such elements, $c_{u,k}=2_+$ for sufficiently large $k \in R$. If there is one such element, then $c_{u,k}=b_{u,k}+1$, and if there is no such element then $c_{u,k}=b_{u,k}$, for sufficiently large $k \in R$. Thus, the sequence $(c_{u,k})$ is still eventually periodic, though with period $PP_0$.

Similar to $(c_{u,k})$, let $(c'_{u,k})$ denote the number of representations of $u*T_k$ where we now also account for all of the representations $u*T_k=(u*T_{k-k'})*T_{i,k'}$ where $k'$ is in the periodic phase of $T_{i,x}$. The same argument above of the periodic case for $i=0$ can similarly be repeatedly applied and extended to the other $m-1$ periodic phases, corresponding to $i=1,\ldots, m-1$, to show that $(c'_{u,k})$ is also periodic. 

Now, all that's left is for us to account for the $m$ finite transient phases. Whether $u*T_k$ is now included depends solely on $c'_{u,k}$, and for each word $T_{i,x} \in \mathcal{U}(\{0,1\})$ from one of the transient phases with $i+x\equiv k \pmod{m}$, whether $u*T_{k-x}$ is in $\mathcal{U}(\{0,1\})$. This is a recurrence relation describing whether $u*T_k$ is in $\mathcal{U}(\{0,1\})$ in terms of which of the $u*T_{k-x}$ are included. However, there are only finitely many such elements $T_{i,x}$, and $c'_{u,k}$ is periodic for large enough $k$, so this recurrence relation has finitely many states. Hence, it must also be eventually periodic. 

We therefore conclude, after fully considering all possibilities, that the values $k$ for which $u*T_k$ is in $\mathcal{U}(\{0,1\})$ are eventually periodic.
\end{proof}

For the special case when $t=0$ and $T=000\cdots$ is simply an infinitely long word of all $0$'s, Theorem \ref{thm: free-periodicity} implies that adding arbitrarily number of $0$'s to the end of a word $u$ will always create new words that are eventually periodically included in $\mathcal{U}(\{0,1\})$. Indeed, we note that Theorems \ref{thm:consec1} and \ref{thm:101} are special cases of this property for $u=\underbrace{0 \cdots 0}_{n}11$ and $\underbrace{0 \cdots 0}_{n}101$, respectively (where $n$ is some nonnegative integer). 

We also see, by Lucas' Theorem \cite{lucas theorem}, that Theorem \ref{thm:one1} shows the existence of this special case of the column phenomenon for words with exactly one $1$ since the binomial coefficients are eventually periodic modulo $2$.

\subsection{Density}

We now consider the density of the Ulam set $\mathcal{U}(\{0, 1\})$ in the following sense.  Let $\mathcal{W}_{2,n}$ denote the set of all binary words of length $n$, and let $V_n=\mathcal{U}(\{0, 1\}) \cap \mathcal{W}_{2,n}$.  We are interested in the quantity $|V_n|/|\mathcal{W}_{2,n}|$, which represents the density of the Ulam set among all binary words of length $n$.

Here are the plotted values of $|V_n|/|\mathcal{W}_{2,n}|$ for $n < 25$, obtained from a computer-generated list of all $6900344$ words in $\mathcal{U}(\{0, 1\})$ of length less than $25$.

\begin{figure}[htp!]
    \centering
    \includegraphics[width=8cm]{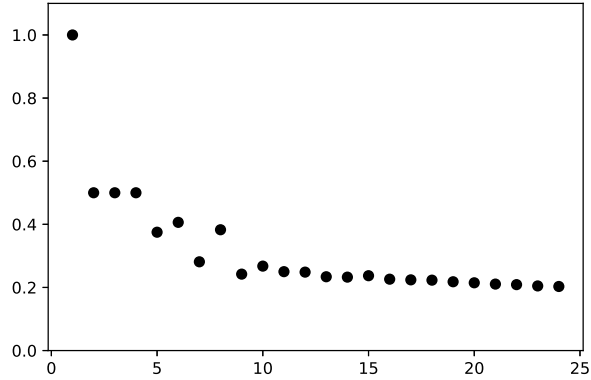}
    \caption{The density of words of size $n$ in $\mathcal{U}(\{0,1\})$.}
\end{figure}

We notice that densities $|V_n|/|\mathcal{W}_{2,n}|$ seem to converge at some real value $r$. In particular, we compute $$|V_{24}|/|\mathcal{W}_{2,24}|=\frac{3406884}{2^{24}}=\frac{851721}{4194304}\approx 0.20306611.$$
\begin{conjecture} \label{conj:density}
There exists some $0<r<1$ such that the density of the Ulam set $\mathcal{U}(\{0, 1\})$ is asymptotically equal to $r$. More formally, we have
$$\lim_{n \to \infty} \frac{|V_n|}{|\mathcal{W}_{2,n}|}=\lim_{n \to \infty}\frac{|V_n|}{2^n}=r.$$
\end{conjecture}

Note that Conjecture \ref{conj:density} is analogous to the problem of finding the density of the classical Ulam sequence, which empirical calculations suggest is approximately $0.074$ \cite{trillion}.

\subsection{$\mathcal{V}$-sets in free groups} While $\mathcal{U}(\{0,1\})$ contains periodic and structured properties, a lot is still unknown about its behavior. However, $\mathcal{V}(\{0,1\})$, the variant of the Ulam set $\mathcal{U}(\{0,1\})$ that allows words to be concatenated with themselves, can be fully characterized.

\begin{theorem} \label{thm:freegroup-Vset}
A word $u$ of length $n$ is in $\mathcal{V}(\{0,1\})$ if and only if $n$ is a power of $2$.
\end{theorem}
\begin{proof}
We induct on $n$. For the base case, $n=1$, we see that both $0$ and $1$ are included in $\mathcal{V}(\{0,1\})$.

For the inductive step, assume the theorem holds for all words of length less than $n$. Then for a word $u$ of length $n$, denote by $u_{i,j}$ the subword of $u$ consisting of the $j$ consecutive letters in $u$ starting after the $i$-th index. If $n=2^{k+1}$ is a power of $2$, then by the inductive hypothesis, all words of length $2^k$ are in $\mathcal{V}(\{0,1\})$, but no words of length strictly between $2^k$ and $2^{k+1}$ are included. Thus, $u$ has a unique representation as the concatenation of smaller words in $\mathcal{V}(\{0,1\})$, $$u=u_{0,2^k}*u_{2^k, 2^k},$$ consisting of concatenating the first half of $u$ with the second half. Therefore, $u$ is included in the $\mathcal{V}$-set.

If $n$ is not a power of $2$ and cannot be written as the sum of two powers of $2$, then $u$ cannot be represented as the concatenation of smaller words in $\mathcal{V}(\{0,1\})$ by our inductive hypothesis. Otherwise, if $n=2^i+2^j$ is not a power of $2$, and so $i\neq j$, but can be written as the sum of two powers of $2$, then by the inductive hypothesis, $$u=u_{0,2^i}*u_{2^i,2^j}=u_{0,2^j}*u_{2^j,2^i}$$
are two different representations of $u$ as the concatenation of smaller terms in $\mathcal{V}(\{0,1\})$. Thus, $u$ is excluded in this case.
\end{proof}

\section{Ulam sets in $\mathbb{Z} \times (\mathbb{Z}/n\mathbb{Z})$}
\label{ZxZn}

We now focus solely on Ulam sets in commutative settings. Suppose we have an abelian group $G$ and a finite initial set $S\subset G$. In \cite{noah}, Kravitz and Steinerberger proved that the choice of the notion of size that determines the order in which the elements are added to $\mathcal{U}(S)$ does not actually affect the set itself. Thus, only the initial set $S$ is needed to construct $\mathcal{U}(S)$. Moreover, it is easy to see that there exists a suitable size function for an initial set $S=\{v_1,\ldots,v_k\}$ (in the sense of Theorem 2 from \cite{noah}) if and only if the equation $$a_1v_1 + \ldots + a_kv_k=0$$ has no solution $(a_1, \ldots, a_k)$ in nonnegative integers where not all $a_1, \ldots, a_k$ are zero.

If we allow negative integers, however, then this equation can have nontrivial solutions in general. Kravitz and Steinerberger also proved in \cite{noah} that the structure of the Ulam set depends only on the solutions to this equation, called the \emph{characteristic equation} of the Ulam set. More precisely, we define two Ulam sets $\mathcal{U}_1=\mathcal{U}(\{v_1,\ldots,v_k\})$ and $\mathcal{U}_2=\mathcal{U}(\{w_1,\ldots,w_k\})$ to be \emph{structurally equivalent} if, for all $(a_1,\ldots,a_k)\in \mathbb{Z}^k$,
$$a_1v_1+\cdots+a_kv_k\in \mathcal{U}_1 \Longleftrightarrow a_1w_1+\cdots+a_kw_k\in \mathcal{U}_2$$ (note that this is related to the idea of Freiman homomorphisms). They proved that two Ulam sets with the same set of solutions of their characteristic equation are structurally equivalent. Note that if we let $L$ be the set of solutions to the characteristic equation, then $L$ must form a subgroup of $\mathbb{Z}^k$. We will therefore call $L$ the \emph{associated lattice} of the Ulam set. Because of this structural equivalence, we will often refer to \emph{the} Ulam set with associated lattice $L$ since all Ulam sets with associated lattice $L$ are structurally equivalent.

The condition on the generators implies that $L \cap \mathbb{Z}_{\ge 0}^k=\{0\}$. Furthermore, we cannot have a vector of the form $e_i-e_j$ with $i\neq j$ in $L$ (where $e_k$ is the canonical basis vector with a $1$ in the $k$-th position and $0$'s everywhere else), since it would imply that $v_i=v_j$. Conversely, for any lattice satisfying those two conditions, we can find an Ulam set having $L$ as its associated lattice: take $G$ to be the quotient group $\mathbb{Z}^k/L$ and $v_1,\ldots,v_k$ to be the images of the canonical basis vectors of $\mathbb{Z}^k$ in this quotient. Thus, there is a one-to-one correspondence (up to structural equivalence) between Ulam sets over commutative groups with $k$ generators and lattices in $\mathbb{Z}^k$ not intersecting $\mathbb{Z}_{\ge 0}^k\setminus \{0\}$ that contains no vector $e_i-e_j$ for $i\ne j$. A particular Ulam set in a particular setting with given associated lattice $L$ is called an \emph{embedding}. 

The case of two generators is particularly interesting: $L$ clearly cannot be two-dimensional, so it must either be one-dimensional or be the zero lattice (which corresponds to the simple case of two linearly independent generators, treated in \cite{noah}). Thus, all the nontrivial cases correspond to one-dimensional lattices in $\mathbb{Z}^2$, which are characterized by a generator $(-a,b)$ where $a,b> 0$ and $a, b$ are not both $1$. In particular, the classical Ulam sequence in $\mathbb{Z}$ with initial set $\{1,2\}$ corresponds to the lattice generated by $(-2,1)$.

If $a$ and $b$ are relatively prime, we can find an embedding in $\mathbb{Z}$ by taking $v_1=b$ and $v_2=a$. The interesting new cases appear when $d=\gcd(a,b)>1$; in this case, there is clearly no embedding in $\mathbb{Z}^m$, since if we have $v_1,v_2 \in \mathbb{Z}^m$ such that $-av_1+bv_2=0$, then we also have $-\frac{a}{d}v_1+\frac{b}{d}v_2=0$. It can, however, be embedded in $\mathbb{Z} \times (\mathbb{Z}/d\mathbb{Z})$. Indeed, let $a'=a/d$, $b'=b/d$ and choose $u,v$ such that $ua'-bv'=1$. Then we can take $v_1=(b',u)$ and $v_2=(a',v)$, and it is straightforward to check that the corresponding Ulam set has associated lattice generated by $(-a,b)$.

Because of this universality, we study, in this section, Ulam sets in $\mathbb{Z} \times (\mathbb{Z}/n\mathbb{Z})$. Note that exploring this setting also originates naturally from complex multiplication. Indeed, as noted in \cite{noah}, Ulam sets arising from multiplication in $\mathbb{R}$ are equivalent to those arising from addition via a logarithm. But in $\mathbb{C}$, the logarithm maps us to $\mathbb{R}\times \mathbb{T}$ (under addition), which suggests the investigation of the discrete analog $\mathbb{Z}\times (\mathbb{Z}/n\mathbb{Z})$.

In this particular setting, the condition that the characteristic equation has no nontrivial solution in nonnegative integers corresponds to the condition that all the initial elements $(x,y) \in \mathbb{Z}\times (\mathbb{Z}/n\mathbb{Z})$ must have positive $x$-coordinate (or all negative in which case we can negate all terms). We will discuss the regularity of these Ulam sets, decomposition of elements in terms of the initial generators, and the possibility of them being finite.

\subsection{Regular Ulam sets}

In this section, we analyze the regularity of Ulam sets. This phenomenon has been studied extensively in the past for Ulam sequences and sets in $\mathbb{Z}^d$ \cite{regularity 4 v, finch regularity, noah, regularity 2 v}, but not in $\mathbb{Z}\times (\mathbb{Z}/n\mathbb{Z})$.

\begin{definition}
\label{regularity}
We say that a set $\mathcal{S} \subset \mathbb{Z}_{\ge 0} \times (\mathbb{Z}/n\mathbb{Z})$ is \emph{regular} if there exists some period $P>0$ such that, for sufficiently large $x$, we have $(x,y)\in \mathcal{S}$ if and only if $(x+P,y)\in \mathcal{S}$. 
\end{definition}

This is an extension of the notion of regularity of a sequence in $\mathbb{Z}$, which corresponds to the case $n=1$ since $\mathbb{Z}$ is equivalent to $\mathbb{Z}\times (\mathbb{Z}/1\mathbb{Z})$. In this particular setting, Finch \cite{finch regularity} proved that an Ulam sequence with finitely many even terms must be regular. In the same paper, he also conjectured that all regular Ulam sequences in $\mathbb{Z}$ satisfy this condition. In $\mathbb{Z}\times (\mathbb{Z}/n\mathbb{Z})$, the situation is more complicated, but the following theorem gives a useful characterization of regular Ulam sets similar to Finch's result. 

\begin{theorem} \label{E condition}
An Ulam set $\mathcal{U}$ in $\mathbb{Z}\times (\mathbb{Z}/n\mathbb{Z})$ is regular if and only if there exists some regular subset $E\subset \mathbb{Z}\times (\mathbb{Z}/n\mathbb{Z})$ such that the sum of two elements of $E$ is never in $E$ and only finitely many elements of $\mathcal{U}$ are not in $E$.
\end{theorem}

\begin{proof} 
We start by showing that the existence of such an $E$ guarantees the regularity of $\mathcal{U}$. Let $M$ be the greatest $x$-coordinate of all elements of $\mathcal{U}\setminus E$, and let $P$ be the period of $E$. For any $x$, let $E_x=\{y\in \mathbb{Z}/n\mathbb{Z}:(x,y)\in E\}$ and $\mathcal{U}_x=\{y\in \mathbb{Z}/n\mathbb{Z}:(x,y)\in \mathcal{U}\}$.

For each $x$, consider the $M$-tuple $L_x=(\mathcal{U}_{x-1},\mathcal{U}_{x-2}, \ldots, \mathcal{U}_{x-M})$. Since there are only $2^{Mn}$ possible such tuples, there is some $x_0>M$ and $k>0$ such that $L_{x_0}=L_{x_0+kP}$. By taking $x_0$ to be sufficiently large, we know, by the regularity of $E$, that $E_{x}=E_{x+kP}$ for every $x\ge x_0$. It then suffices to show for $x\ge x_0$ that $E_x$ and $L_x$ together uniquely determine $\mathcal{U}_{x}$, since applying it to $x_0, x_0+1,x_0+2, \ldots$ will prove that $\mathcal{U}$ is regular.

However, we see that this is a consequence of the condition that the sum of two elements of $E$ is never in $E$. In particular, if $y\in E_x$, then any representation of $(x,y)$ as a sum of previous elements of $\mathcal{U}$ must use one of the finitely many elements from $\mathcal{U}\setminus E$. Thus, $(x,y)$ is in $\mathcal{U}$ if and only if exactly one of $(x,y)-p$ is in $\mathcal{U}$ for $p\in \mathcal{U}\setminus E$. Moreover, if $y\not\in E_x$ then $(x,y)\not\in \mathcal{U}$ because $x\ge x_0>M$. This is enough to show that $E_x$ and $L_x$ together uniquely determine $\mathcal{U}_{x}$.

Now, for the converse, we only need to show that there exists such an $E$ given that $\mathcal{U}$ is regular with period $P$. It suffices to take $E$ to be the periodic section of $\mathcal{U}$. We see that, with this construction, $E$ is regular and that only finitely many elements of $\mathcal{U}$ are not in it. Now suppose that $(x_1,y_1)$ and $(x_2,y_2)$ are both in $E$, with $x_1\le x_2$. Then, since $E$ is the periodic section, we have the two distinct representations $$(x_1+x_2+3P,y_1+y_2)=(x_1,y_1)+(x_2+3P,y_2)=(x_1+P,y_1)+(x_2+2P,y_2),$$ 
so $(x_1+x_2+3P,y_1+y_2)\not\in \mathcal{U}$ and $(x_1+x_2,y_1+y_2)\not\in \mathcal{U}$. This satisfies all the properties of $E$, so the proof is complete.
\end{proof}

Note that, when $n=1$, the backwards direction of Theorem \ref{E condition} recovers a known result of Ross (the first part of Theorem 6.3.2 from \cite{ross-thesis}). Finch's result \cite{finch regularity} is also a consequence of Theorem \ref{E condition} for the case $n=1$ when $E$ is the set of odd numbers. More generally, we have the following: 

\begin{corollary}
Let $\mathcal{U}$ be an Ulam set in $\mathbb{Z}\times (\mathbb{Z}/n\mathbb{Z})$ with two generators, $v_1$ and $v_2$, and associated lattice generated by $(-a,b)$. Then the following must be true:
\begin{enumerate}
    \item If $\mathcal{U}$ contains finitely many elements with an even $x$-coordinate, it is regular.
    \item If $n$ is even and $\mathcal{U}$ contains finitely many elements with an even $y$-coordinate, it is regular.
    \item If $a$ is even and $\mathcal{U}$ contains finitely many elements of the form $2xv_1+yv_2$, it is regular.
\end{enumerate}
\end{corollary}

\begin{proof}
Take $E$, in Theorem \ref{E condition}, to be
\begin{enumerate}
    \item the set of elements with an odd $x$-coordinate.
    \item the set of elements with an odd $y$-coordinate.
    \item the set of elements of the form $xv_1+yv_2$, where $x$ is odd.
\end{enumerate}
It is straightforward to check that these choices of $E$ satisfy the conditions of Theorem \ref{E condition}, so $\mathcal{U}$ must be regular.
\end{proof}

These choices of $E$, especially the third one, seem to readily apply for most regular Ulam sets that we have encountered. More precisely, computations have shown that the following conjecture seems to hold:

\begin{conjecture}
\label{regularity conjecture}
Let $a>2$ be even and $b$ be sufficiently large. An Ulam set with generators $v_1$ and $v_2$ and associated lattice generated by $(-a,b)$ contains finitely many elements of the form $2xv_1+yv_2$ and is therefore regular.
\end{conjecture}

This is in agreement with Finch's conjecture \cite{finch regularity} about exactly which sequences in $\mathbb{Z}$ are regular (corresponding to the case where $a$ and $b$ are relatively prime). The main difficulty behind proving Conjecture \ref{regularity conjecture} is that when $a$ becomes large, there can be many elements outside of $E$ (of the form $2xv_1+yv_2$), making it hard to characterize the elements of $\mathcal{U}$.

We now focus on the case $a=2$. Note that this case is not included in Conjecture \ref{regularity conjecture} because when $b$ is a power of two, there are infinitely many elements of the form $2xv_1+yv_2$. For all other values of $b$, however, there are only two such elements. This generalizes Schmerl and Spiegel's theorem \cite{regularity 2 v} that states that an Ulam sequence in $\mathbb{Z}$ generated by $a=2$ and $b>3$, where b is odd, has exactly two even terms. Our result allows $b$ to be even (in which case the Ulam set cannot be embedded in $\mathbb{Z}$). 

\begin{theorem}
\label{regularity a=2}
Let $b>3$, and let $\mathcal{U}$ be an Ulam set generated by two elements $v_1$ and $v_2$ with associated lattice generated by $(-2,b)$. If $b$ is not a power of $2$, then $\mathcal{U}$ contains exactly two terms of the form $2xv_1+yv_2$: $v_2$ and $(b+1)v_2=2v_1+v_2$.
\end{theorem}

The beginning of our proof is largely based off of the original proof where $b$ is odd. We will need to compute the first few elements of $\mathcal{U}$. Note that, by adding a suitable multiple of $-2v_1+bv_2$, each element of $\mathcal{U}$ can be expressed uniquely as $x{v_1}+y{v_2}$ for $x \in \{0,1\}$ and $y \ge 0$.

\begin{lemma}\label{first elts}
The elements of $\mathcal{U}$ of the form $v_1+yv_2$ for $0 \le y \le 3b+2$ are those with exactly one of the following:
\begin{enumerate}
    \item $0 \le y \le b$,
    \item $b< y \le 2b$, where $y \equiv b \pmod{2}$,
    \item $2b<y\le 3b+2$, where $y\equiv -1,0 \pmod{4}$.
\end{enumerate} 
Furthermore, the elements of $\mathcal{U}$ of the form $yv_2$ for $0 \le y \le 3b+2$ are those with $y=1$ or $y=b+1$.
\end{lemma}

\begin{proof}
This is straightforward with strong induction on $y$, and where we need $b > 3$. For the base case, $y=0$, we see that $v_1$ is a generator and $0$ is clearly not in $\mathcal{U}$.

For the inductive step, fix $y$ and assume that the lemma holds for all smaller values of $y$. If $1 \le y \le b$ then $v_1+yv_2$ has the unique representation $$v_1+yv_2=(v_1+(y-1)v_2)+v_2.$$ For $b < y \le 2b$, we have the representation $$v_1+yv_2=(b+1)v_2+(v_1+(y-b-1)v_2).$$ If $y \equiv b \pmod{2}$, this representation is unique. Otherwise, we also have $$v_1+yv_2=(v_1+(y-1)v_2)+v_2.$$ Now let $2b < y \le 3b+2$. If $y \equiv 0 \pmod{4}$, then $v_1+y{v_2}$ has the unique representation $$v_1+y{v_2}=(v_1+(y-1)v_2)+v_2.$$ If $y \equiv -1 \pmod{4}$, we have the unique representation $$v_1+yv_2=(v_1+(y-b-1)v_2)+(b+1)v_2.$$ If $y\equiv 1 \pmod{4}$, we have two representations: $$v_1+y{v_2}=(v_1+(y-b-1)v_2)+(b+1)v_2= (v_1+(y-1)v_2)+v_2.$$ Finally, if $y\equiv 2 \pmod{4}$ then we have no representation of $v_1+y{v_2}$.

Continuing the induction for elements of the form $yv_2$, we see that if $1< y \le b$, then $yv_2$ has no representation. The element $v_2(b+1)=2v_1+v_2$ has the unique representation $v_2(b+1)=(v_1+v_2)+v_2$, so it is in $\mathcal{U}$. The element $y(b+2)=(b+1)v_2+v_2=(v_1+2v_2)+ v_1$, however, can be expressed in two ways.

For $b+2<y \le 3b-3$, we also have two representations. Precisely, if we let $$i=\left \lfloor \frac{y-b-1}{2} \right \rfloor \quad \text{and} \quad j=\left \lceil \frac{y-b+1}{2} \right\rceil,$$ 
then we see that $$yv_2=(v_1+iv_2)+(v_1+jv_2)=(v_1+(i-1)v_2)+(v_1+(j+1)v_2).$$ 
Finally, for $3b-2 \le y \le 3b+2$, we have the following two representations $$yv_2=(v_1+(y-b-\delta)v_2)+(v_1+\delta v_2)=(v_1+(y-b-\delta-2)v_2)+(v_1+(\delta+2) v_2)$$ where we choose $\delta$ to be either $1$ or $2$ depending on the parity of $b$ and $y$. 
\end{proof} 

\begin{proof}[Proof of Theorem \ref{regularity a=2}]
Assume for the sake of contradiction that there is another element $$s=(N+b)v_2=2v_1 + Nv_2$$ of $\mathcal{U}$ and take such an element with minimal $N$. We consider the binary sequence $(a_i)$, where $a_i \in \mathbb{Z}/2\mathbb{Z}$ and
\[ 
a_i= \begin{cases} 
      0 & \text{if $v_1+iv_2$ is not in $\mathcal{U}$}\\
      1 & \text{otherwise}
      \end{cases}
\]
keeps track of the elements of the form $v_1+iv_2$ that are in the set. Since every representation of an element $v_1+iv_2$ as a sum of smaller terms in $\mathcal{U}$ must use an element of the form $iv_2$, we have, for $b<i\le N$, that $a_i=1$ if and only if exactly one of $a_{i-1}$ and $a_{i-b-1}$ equals $1$. This gives us the recursive formula $$a_i = a_{i-1}+a_{i-b-1},$$ which holds for $b<i\le N$. From now on, any evaluation of the $a_i$'s will be understood to be taken modulo $2$. We can now use this recursive relation to show that there are no gaps in $(a_i)$ of length $b+1$.

Precisely, we claim that, for every $b \le t \le N$, we have $1 \in \{a_{t-i}:0\le i \le b\}$. Suppose, for the sake of contradiction, that $t$ is a minimal counterexample. In this case, $a_{t-1}=a_t=0$. Clearly, $t>b$, so by the recursive formula we have $a_{t-b-1}=a_{t-1}-a_t=0$, but that means that $t-1$ is actually a smaller counterexample, which is a contradiction.

Applying this to $t = N$, we obtain that there is some $0\le i \le b$ such that $a_{N-i}=1$. Then, by Lemma \ref{first elts}, we have the following representation of $s$ as the sum of two distinct elements of $\mathcal{U}$: $$s=(s-v_1-iv_2)+(v_1+iv_2).$$ Note that the elements are distinct because otherwise $s=2(v_1+iv_2)=(b+i)v_1$, contradicting Lemma \ref{first elts}.
Since the representation must be unique, the choice of $i$ is also unique. Thus, $a_{N-j}=0$ for $0\le j \le b$ and $j \ne i$.

Using the recursive formula backwards, we can determine $a_{N-j}$ for $b<j\le 2b$. We find that $a_{N-j}=1$ if and only if $j=i+b+1$ or $j=i+b$ (for $i>0$). We now have two cases to treat separately:

\textbf{Case 1: $i>0$}. In this case, both $s-v_1-(i+b)v_2$ and $s-v_1-(i+b+1){v_2}$ are in $\mathcal{U}$. However, by Lemma \ref{first elts}, one of ${v_1}+(i+b){v_2}$ and ${v_1}+(i+b+1){v_2}$ is in $\mathcal{U}$. Thus, we have a new representation $$s=(s-v_1-(i+b+\delta){v_2})+(v_1+(i+b+\delta){v_2}),$$ for either $\delta=0$ or $\delta=1$. If the summands are distinct, then this is a contradiction. Thus the summands must be the same, so $b+i+\delta=N-(b+i+\delta)$. We know, however, that $a_\ell=0$ for $N-i-b<\ell<N-i$; thus, $a_\ell=0$ for $i+b+2\delta < \ell < i+2b+2\delta$. This, however, contradicts Lemma \ref{first elts}, so Case 1 is impossible.

\textbf{Case 2: $i=0$}. This is where the proof diverges from that given by Schmerl and Spiegel \cite{regularity 2 v}. Our goal is to find some $k>0$ for which $a_k=a_{N-k}=1$, which will give us a second representation of $s$ as a sum of two distinct elements of $\mathcal{U}$. The recursive formula suggests arranging the $a_i$'s into an array of height $b+1$. This way, every number (except for the upper row) is the sum of the number to its left and the number above modulo $2$. It is not surprising with this rule to see the beginning of Pascal's triangle modulo $2$ appearing. (Note that this is not the first time that it appears in the study of Ulam sets; there are similar connections in Theorem \ref{thm:one1} and \cite{regularity 4 v, V-sequences}.)

\begin{table}[hbt!]
\centering
\begin{tabular}{llllll}
1 & 0 & 0 & 0 & 0 & 1 \\
1 & 1 & 1 & 1 & 1 & 0 \\
1 & 0 & 1 & 0 & 1 & $\vdots$ \\
1 & 1 & 0 & 0 & 1 &   \\
1 & 0 & 0 & 0 & 1 &   \\
1 & 1 & 1 & 1 & 0 &   \\
1 & 0 & 1 & 0 & 0 &   \\
1 & 1 & 0 & 0 & 0 &   \\
1 & 0 & 0 & 0 & 0 &   \\
1 & 1 & 1 & 1 & 1 &   \\
1 & 0 & 1 & 0 & 1 &   \\
1 & 1 & 0 & 0 & 1 &   \\
1 & 0 & 0 & 0 & 1 &   
\end{tabular}
\caption{The first few columns of the array formed for $b=12$. We see the beginning of Pascal's triangle in the first $5$ columns.}
\end{table}

For $b=2^e\cdot c$ for some odd value $c$, these observations suggest trying to prove the following: 

For $0 \le q\le 2^e$ and $0\le r <b+1$, we have 

 \[
a_{q(b+1)+r}= \begin{cases} 
      1 & \text{if $r=q=0$}\\
      \binom{r-1+q}{q} \pmod{2} & \text{otherwise}
      \end{cases}.
\]

We prove this claim using strong induction and our recursive formula. The base case $q=0$ has already been treated in Lemma \ref{first elts}. For our inductive step, suppose that $q>0$ and our formula holds for smaller $q$ as well as for the same value $q$ but with smaller values $r$. If $r>0$, we have $$a_{q(b+1)+r}= a_{(q-1)(b+1)+r}+a_{q(b+1)+r-1}= {r-2+q \choose q-1}+{r-2+q \choose q}={r-1+q \choose q}.$$

If $r=0$, we can assume that $q>1$ (the case $q=1$ and $r=0$ has also been treated in Lemma \ref{first elts}). We then have $$a_{q(b+1)}=a_{(q-1)(b+1)}+a_{(q-1)(b+1)+b}=0+{b+q-2 \choose q-1}.$$ 
Since $q-1<2^e$, this binomial coefficient is zero modulo $2$ by Lucas' Theorem \cite{lucas theorem}, so $a_{q(b+1)}=0={q-1 \choose q}$, and the proof of our claim is complete.

Using this formula, we see that $$a_{(2^e+1)(b+1)}=a_{2^e(b+1)}+a_{2^e(b+1)+b}=0+{b-1+2^e \choose 2^e}=0+{(c+1)2^e-1 \choose 2^e}=1,$$ by Lucas' Theorem. 

We now work backwards to show that $a_{N-(2^e+1)(b+1)}=1$. Specifically, we claim that, for $0 \le q\le b$ and $0 < r < b+1-q$, $a_{N-q(b+1)-r}=0$ and $a_{N-q(b+1)}=1$.

This similarly follows via strong induction by using our recursive formula. For the base case, $q=0$, we already know that $a_{N-r}=0$ if $0<r<b+1$ and $a_{N-r}=1$ if $r=0$. For larger values $0 < q\le b$ and $q < r < b+1-q$, we see that $$a_{N-q(b+1)-r}=a_{N-(q-1)(b+1)-r}-a_{N-(q-1)(b+1)-(r+1)}=0,$$ by our inductive hypothesis and our recursive formula.
Finally, we have $$a_{N-q(b+1)}=a_{N-(q-1)(b+1)}-a_{N-(q-1)(b+1)-1}=1.$$

We therefore know that if $b>2^e$, then $a_{N-(2^e+1)(b+1)}=1$. Thus, we only need to check that $N-(2^e+1)(b+1) \ne (2^e+1)(b+1)$ to show that there is another representation of $s$ as a sum of two distinct elements of $\mathcal{U}$. However, we know that $a_{(2^e)(b+1)}=0$ and $a_{N-(2^e+2)(b+1)}=1$ (because $b\ne 2^e$), which is a contradiction if $N=2(2^e+1)(b+1)$. This completes the proof.
\end{proof}

Note that, since there is no new term of the form $yv_2$ by Lemma \ref{first elts}, the recursive formula we found for the sequence $(a_i)$ actually holds forever. The sequence given by this recursive formula has actually already been studied, for instance in \cite{binary sequence}. The length of its period is given by OEIS Sequence A046932. 

The proof of Theorem \ref{regularity a=2} fails if $b$ is a power of two. In this case, there are infinitely many elements of the form $yv_2$. However, $\mathcal{U}$ is still regular, and there even exists a simple closed formula characterizing the elements of $\mathcal{U}$. We can arrange the sequence $(a_i)$ into an array of height $b+1$ as before, which displays an infinite pattern as shown in Table \ref{array01}.

\begin{table}[hbt!]
\centering
\begin{tabular}{llllllllllllllllllll} 
1 & 0 & 0 & 0 & 0 & 0 & 0 & 0 & 0 & 0 & 0 & 0 & 0 & 0 & 0 & 0 & 0 & 0 & 0 &          \\
1 & 1 & 1 & 1 & 1 & 1 & 1 & 1 & 1 & 0 & 1 & 0 & 1 & 0 & 1 & 0 & 1 & 0 & 1 &          \\
1 & 0 & 1 & 0 & 1 & 0 & 1 & 0 & 1 & 0 & 1 & 0 & 1 & 0 & 1 & 0 & 1 & 0 & 1 &          \\
1 & 1 & 0 & 0 & 1 & 1 & 0 & 0 & 1 & 0 & 0 & 0 & 1 & 0 & 0 & 0 & 1 & 0 & 0 &          \\
1 & 0 & 0 & 0 & 1 & 0 & 0 & 0 & 1 & 0 & 0 & 0 & 1 & 0 & 0 & 0 & 1 & 0 & 0 & $\cdots$ \\
1 & 1 & 1 & 1 & 0 & 0 & 0 & 0 & 1 & 0 & 1 & 0 & 0 & 0 & 0 & 0 & 1 & 0 & 1 &          \\
1 & 0 & 1 & 0 & 0 & 0 & 0 & 0 & 1 & 0 & 1 & 0 & 0 & 0 & 0 & 0 & 1 & 0 & 1 &          \\
1 & 1 & 0 & 0 & 0 & 0 & 0 & 0 & 1 & 0 & 0 & 0 & 0 & 0 & 0 & 0 & 1 & 0 & 0 &          \\
1 & 0 & 0 & 0 & 0 & 0 & 0 & 0 & 1 & 0 & 0 & 0 & 0 & 0 & 0 & 0 & 1 & 0 & 0 &     
\end{tabular}
\caption{The first few columns of the array formed by the sequence $(a_i)$ for $b=8$. We see that it consists of the beginning of Pascal's triangle repeated, with the odd columns deleted after the first block.}
\label{array01}
\end{table}

\begin{theorem}
\label{b=2^e}
For $e>1$, let $b=2^e$ and let $\mathcal{U}$ be the Ulam set generated by two elements, $v_1$ and $v_2$, with associated lattice generated by $(-2,b)$. Then $yv_2$ is in $\mathcal{U}$ if and only if $y=1$ or $y=(b+1+kb(b+1))$ for some integer $k \ge 0$.
Moreover, if $i=p\cdot b(b+1) + q\cdot(b+1) +r$ with $p\ge 0$, $0\le q <b$, and $0\le r < b+1$, then (defining $a_i$ as in Theorem \ref{regularity a=2}) we have
\[ 
a_i= \begin{cases} 
      1 & \text{if $p=q=r=0 $}\\
      \binom{r-1+q}{q} \pmod{2} & \text{if $r>0$ and ($p=0$ or $q$ is even)} \\
      0 & \text{otherwise}
      \end{cases}.
\]
\end{theorem}
\begin{proof} We proceed with strong induction and casework. Since there are numerous cases with technical details, we go through them quickly. The case $p=0$ follows from the first claim in Case~2 of the proof of Theorem \ref{regularity a=2}, so suppose $p>0$. 

If $r=0$ and $q=1$, then $$a_{i-1}=a_{(p-1) b(b+1) + b}=\binom{b-1}{0}=1 \quad \text{and} \quad a_{i-(pb(b+1)+(b+1))}=a_0=1,$$ so we have two distinct representations; hence, $a_i=0$. If $r=0$ and $q> 1$, then $$a_{i-1}=a_{pb(b+1) + (q-1)(b+1) +b}={b+q-2 \choose q-1}=0$$ and
$$a_{i-(kb(b+1)+b+1)}=a_{(p-k) b(b+1) + (q-1)(b+1)}=0,$$ which implies that no representation exists, so $a_i=0$. If $r=0$ and $q=0$, then
$$a_{i-1}=a_{(p-1)b(b+1) + (b-1)(b+1) +b}={2b-2 \choose b-1}=0$$
and
$$a_{i-(kb(b+1)+b+1)}=a_{(p-k-1) b(b+1) + (b-1)(b+1)}=0,$$ so, once again, there is no representation, meaning $a_i=0$.

Now suppose that $r>0$. If $q$ is odd, then $a_{i-1}=0$ and $a_{i-(kb(b+1)+b+1)}= a_{i-(b+1)}$ for all $k$. This means that if a representation exists, then there must be at least two, so $a_i=0$.
If $q=0$ and $r=1$, then $a_{i-1}=0$ and $a_{i-(kb(b+1)+b+1)}=0$ except for when $k=(p-1)$, where $$a_{(b-1)(b+1)+1}=\binom{(b-1)(b+1)}{(b-1)(b+1)}=1,$$ 
so $a_i=1$. If $q=0$ and $r>1$, then both $a_{i-1}=0$ and $a_{i-(kb(b+1)+b+1)}=0$ for any $k$, so $a_i=1$ because there is, once again, a unique representation.

If $q>0$ is even, then $a_{i-(kb(b+1)+b+1)}=0$ except for when $k=p$, in which case $$a_{i-(kb(b+1)+b+1)}=a_{(q-1)(b+1)+r}.$$ We therefore have
$$a_i=a_{(q-1)(b+1)+r}+a_{pb(b+1) + q(b+1) +r-1}={q+r-2 \choose q-1}+{q+r-2 \choose q}={q+r-1 \choose q}.$$

All that's left to be shown is that the only elements of the form $yv_2$ in $\mathcal{U}$ are the ones we claimed. It is clear that $v_2\in \mathcal{U}$ and we can check that, for $k\ge 0$, the only representation of $(b+1+kb(b+1))v_2=2v_1+(1+kb(b+1))v_2$ is $v_1+(v_1+(1+kb(b+1))v_2)$.

Now suppose that some other element $2v_1+Nv_2$ has a unique representation. Consider expressing $N$ as $N=p_N\cdot b(b+1) + q_N\cdot(b+1) +r_N$, with $p_N\ge 0$, $0\le q_N <b$, and $0\le r_N < b+1$. If $r_N> 1$, then $a_{p_N b(b+1)+r_N-1}=a_{q_N(b+1)+1}=1$, so we have the representation $$2v_1+Nv_2=(v_1+(q_N(b+1)+1)v_2)+(v_1+(p_N b(b+1)+r_N-1)v_2).$$
Moreover, for every $k$, $a_{2k(b+1)+1}=a_{2k(b+1)+2}=1$. Let $\ell$ be the even number between $0$ and $2(b+1)-1$ such that $N-\ell$ is of the form $2k(b+1)+1$ or $2k(b+1)+2$. Then we also have the representation $2v_1+Nv_2=(v_1+\ell v_2)+(v_1+(N-\ell)v_2)$. The two representations we found must actually be equal, so $\ell=q_N(b+1)+1$ or $\ell=p_N b(b+1)+r_N-1$. 

In the first case, since $\ell<2(b+1)$, we have $q_N=0$ or $q_N=1$. If $q_N=0$, then we have the representation $2v_1+Nv_2=v_1+(v_1+Nv_2)$ which is distinct from the previous one. Thus, $q_N=1$. If $r_N>3$, then we have the new representation $$2v_1+Nv_2=(v_1+(b+4)v_2)+(v_1+(p_N b(b+1)+r_N-3)v_2).$$ If $r_N\in \{2,3\}$, then we have the other representation $$2v_1+Nv_2=(v_1+4v_2)+(v_1+(N-4)v_2).$$ In the second case, $p_N=0$, so $\ell=r_N-1<b$. Then, we have the new representation $$2v_1+Nv_2=(v_1+(r_N-2)v_2)+(v_1+(q_N(b+1)+2)v_2),$$ which is distinct from the previous one. This concludes the case $r_N>1$.

Now if $r_N=q_N=0$ and $p_N>1$, we have the following two distinct representations: $$2v_1+Nv_2=(v_1+((b-1)(b+1)+1)v_2)+(v_1+((p_N-1)b(b+1)+b)v_2)$$
and
$$2v_1+Nv_2=(v_1+2bv_2)+(v_1+((p_N-1)b(b+1)+(b-2)(b+1)+2)v_2).$$

If $r_N=0$ and $q_N\ne 0$, then $$2v_1+Nv_2=(v_1+((q_N-1)(b+1)+1)v_2)+(v_1+(p_Nb(b+1)+b)v_2).$$ 
If $q_N$ is odd, then we also have the representation $2v_1+Nv_2=(v_1+bv_2)+(v_1+(N-b)v_2)$, which is distinct from the other. If $q_N$ is even, then $2v_1+Nv_2=(v_1+2bv_2)+(v_1+(N-2b)v_2)$ is also a distinct representation.

Finally, consider the case $r_N=1$. We have already shown that if $q_N=0$, there is a unique representation. Thus, suppose $q_N>0$. If $q_N$ is even, then we have the representations
$$2v_1+Nv_2=v_1+(v_1+Nv_2)$$
and
$$2v_1+Nv_2=(v_1+((b+1)+q_N+1)v_2)+(v_1+(p_Nb(b+1)+(q_N-2)(b+1)+b-q_N+1)v_2).$$
If $q_N$ is odd, we have $$2v_1+Nv_2=(v_1+(q_N+1)v_2)+(v_1+(p_Nb(b+1)+(q_N-1)(b+1)+b-q_N+1)v_2)$$
and
$$2v_1+Nv_2=(v_1+(q_N+2)v_2)+(v_1+(p_Nb(b+1)+(q_N-1)(b+1)+b-q_N)v_2),$$ which are again different representations. We have now exhausted all cases, and the proof is complete.
\end{proof}

We also observe that the Ulam sets with associated lattice generated by $(-n,n)$, where $n=2^e+2$ for some integer $e\ge 3$, appear to be regular.
In this case, we notice that there seems to be no element of the form $2xv_1+2yv_2$. Moreover, based on computations, we conjecture a full characterization of all elements of the form $2xv_1+yv_2$ that are in $\mathcal{U}$. In particular, we believe that $2xv_1+yv_2$ is in $\mathcal{U}$ if and only if $$2xv_1+yv_2=(p(2^e+1)(2^e+2)+q(2^e+2)+r)v_1+v_2$$ for some $p\ge 0$, $0\le q <2^e+1$, and $0\le r <2^e+2$ even, with $p$, $q$, and $r$ satisfying one of the following conditions:

\begin{enumerate}
    \item $q=r=0$
    \item $p=0, q=1, r=0$
    \item $t,q<2^e$ and one of the following is true, where $t=2^e+1-r$.
        \begin{itemize}
            \item $p=0$ and ${q+t \choose t}\equiv 1 \pmod{2}$
            \item $p>0$ and $q+t=2^e-1$
            \item $p>0$, $q+t=2^e-3$, and $p \equiv 1 \pmod{4}$
        \end{itemize}  
\end{enumerate}

Symmetric conditions would also apply for elements of the form $xv_1+2yv_2$. If our conjecture is correct, then any representation of the remaining elements of $\mathcal{U}$ of the form $(2x+1)v_1+(2y+1)v_2$ must use one element of the form $xv_1+2yv_2$ and one of the form $2xv_1+yv_2$. Thus, there likely exists some full characterization of all the elements of $\mathcal{U}$ derivable through much detailed casework, similar to Theorem \ref{b=2^e}.

\subsection{Keeping track of the contribution of each generator}

In an Ulam set, any element decomposes uniquely as a sum of two previous elements, which themselves uniquely decompose. Repeating this process, we can therefore express any term canonically as a linear combination of the original generators. Recall that we defined the function $\alpha_\mathcal{U}$ as follows:

\begin{definition} \label{def:alpha}
Let $\mathcal{U}$ be an Ulam set with the generators $v_1$ and $v_2$. Define the function $\alpha_{\mathcal{U}} : \mathcal{U} \to \mathbb{Z}_{\ge 0}^2$ recursively as follows:
\begin{itemize}
    \item Set $\alpha_{\mathcal{U}}(v_1)=(1,0)$, $\alpha_{\mathcal{U}}(v_2)=(0,1)$.
    \item Any other $u \in \mathcal{U}$ can be written uniquely as $u=u_1+u_2$ for some $u_1,u_2 \in \mathcal{U}$. Then set $\alpha_{\mathcal{U}}(u)=\alpha_{\mathcal{U}}(u_1)+\alpha_{\mathcal{U}}(u_2)$.
\end{itemize}
When the context is clear, we refer to $\alpha_{\mathcal{U}}$ simply as $\alpha$.
\end{definition} 

 If we let $\alpha(u)=(\alpha_1(u),\alpha_2(u))$, it is easy to check that $u=\alpha_1(u)v_1+\alpha_2(u)v_2$, as we would expect. Thus, we can recover $u$ from $\alpha(u)$, but $\alpha(u)$ also contains more information about how $u$ is formed.

Plotting $\alpha$ in the plane leads to a very surprising observation; asymptotically, all the points seem to cluster around a straight line passing through the origin. 

\begin{conjecture} \label{alpha1/alpha2 conjecture}
Let $\mathcal{U}$ be an Ulam set with two generators and nonzero associated lattice and let $\alpha(u)=(\alpha_1(u),\alpha_2(u))$. Then there exist some number $r\in \mathbb{R}$ such that, for every $\varepsilon>0$, we have $\left|\frac{\alpha_1(u)}{\alpha_2(u)}-r\right|<\varepsilon$ for all but finitely many $u\in \mathcal{U}$.
\end{conjecture} 

\begin{figure}[htp!]
    \centering
    \includegraphics[width=9cm]{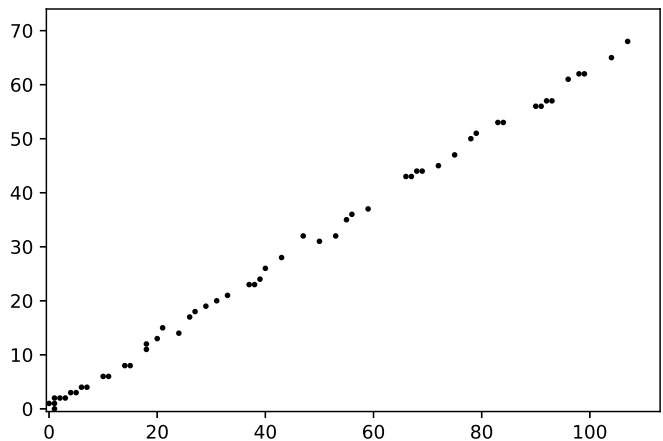}
    \caption{The image of $\alpha$ for the original Ulam sequence $\mathcal{U}(\{1,2\})$.}
\end{figure}

This is another example of a surprising property of Ulam sets and sequences about which very little is known and it may be related to other important phenomenons such as the ``hidden signal'' found by Steinerberger \cite{hidden signal} or the Rigidity Conjecture \cite{rigidity}. Although Conjecture \ref{alpha1/alpha2 conjecture} does not tell us which elements can be in the Ulam set, it tells us how each element is formed. We now prove the somewhat weaker result that the ratio $\frac{\alpha_1(u)}{\alpha_2(u)}$ cannot be arbitrarily small or large.

\begin{theorem}
\label{alpha1/alpha2}
Let the associated lattice of $\mathcal{U}$ be generated by $(-a,b)$ and let $\mathcal{U}'$ be the image of $\alpha$ in $\mathbb{Z}_{\ge 0}^2$. Then for every $(x,y) \in \mathcal{U}'$ other than $(0,1)$ and $(1,0)$, we have $y \le b(2x-1)$ and $x \le a(2y-1)$.
\end{theorem}

\begin{proof}
Note that the only element with $x$-coordinate $0$ is $(0,1)$ while the elements $(1,0), \ldots, (1,b)$ are in $\mathcal{U}'$. However, $(1,b+1)$ is not included because $v_1+(b+1)v_2=(v_1+bv_2)+v_2=v_1+(av_1+v_2)$. For $y>b+1$, the only way to express $(1,y)$ as a sum of elements of $\mathcal{U}'$ is to use $(0,1)$, so we would need $(1,y-1)$ to also be in $\mathcal{U}'$. Thus, for $y>b+1$, $(1,y)$ is not in $\mathcal{U}'$.

Now let's consider the other values of $x$. Since $(av_1+v_2)=(b+1)v_2 \in \mathcal{U}$, we cannot have more than $b+1$ consecutive elements of $\mathcal{U}'$ in a vertical line. Indeed, if $(x,y), (x,y+1), \ldots, (x,y+b)$ are all in $\mathcal{U}$, then $(x,y+b+1)$ cannot be in $\mathcal{U}$ because $$xv_1+(y+b+1)v_2=(xv_1+(y+b)v_2)+v_2=(xv_1+yv_2)+(b+1)v_2.$$

We will now prove the result by inducting on $x$. The cases $x=0$ and $x=1$ have already been treated, so suppose $x\ge 2$ and that the result holds for all smaller $x$. 

Assume $(x,y)\in \mathcal{U}'$ has the representation $(x,y)=(x_1,y_1)+(x_2,y_2)$, where $x_1\le x_2$. If $x_1\ge 1$, then we have $y_1\le b(2x_1-1)$ and $y_2\le b(2x_2-1)$ by our inductive hypothesis, so $$y=y_1+y_2\le b(2x-2).$$ This means that if $y>b(2x-2)$, then $x_1=0$, so $y_1=1$ and $(x,y-1)=(x_2,y_2)\in \mathcal{U}'$. Hence, if $y=b(2x-2)+k$, then $(x,y-1), (x,y-2), \ldots, (x,y-k)$ are also in $\mathcal{U}$. If $k\ge b+1$, then we have more than $b+1$ consecutive elements, which is impossible as we saw earlier. Thus, $k \le b$ and $y \le b(2x-1)$ except for when $(x,y)=(0,1)$.

By symmetry, we also have that $x \le a(2y-1)$ except for when $(x,y)=(1,0)$, so the proof is complete.
\end{proof}

This result gives us the bounds $\frac{y}{x} \le 2b$ and $\frac{x}{y} \le 2a$, which holds for all but finitely many elements. 

We now focus on the case where the associated lattice is generated by $(-n,n)$. This case is particularly interesting because, if Conjecture \ref{alpha1/alpha2 conjecture} is true, then $r=1$ by symmetry. We note that this is the only case where there seems to be a simple expression for the exact value of $r$. Moreover, we can take advantage of the symmetry to improve the bound of the previous theorem to $\frac{y}{x} < n+1$.

\begin{theorem}
\label{alpha1/alpha2 a=b}
When $\mathcal{U}$ has associated lattice generated by $(-n,n)$, then for any $(x,y)\in \mathcal{U}'$ other than $(0,1)$ and $(1,0)$, we have $y < (n+1)x$ and $x<(n+1)y$.
\end{theorem}

We first require a preparatory lemma.

\begin{lemma}
\label{x=y (mod n)}
Let $\mathcal{U}$ have associated lattice generated by $(-n,n)$. If $(x,y)\in \mathcal{U}'$ and $x\equiv y \pmod{n}$, then $x=y$.
\end{lemma} 

\begin{proof} 
By symmetry, we know that $(y,x)$ is in $\mathcal{U}'$. Thus, there exists some $u_1, u_2 \in \mathcal{U}$ such that $\alpha(u_1)=(x,y)$ and $\alpha(u_2)=(y,x)$. But then we have $$u_1=xv_1+yv_2=yv_1+xv_2+(x-y)(v_1-v_2)=yv_1+xv_2=u_2$$ since $n|(x-y)$. This means that $(x,y)=\alpha(u_1)=\alpha(u_2)=(y,x)$; hence $x=y$.
\end{proof}

\begin{proof}[Proof of Theorem \ref{alpha1/alpha2 a=b}]
Suppose, for the sake of contradiction, that some $u \in \mathcal{U}$ other than $v_2$ is mapped by $\alpha$ to $(x,y)$ with $y\ge (n+1)x$. Take such a $u$ which lexicographically minimizes $(x,y)$.

Since the sum of two vectors below the line $y=(n+1)x$ stays below the line, the representation of $u$ must use another element that is mapped above the line. Note that the only choice is $v_2$, since $u$ corresponds to the minimal such $(x,y)$. This means that $u-v_2$ is in $\mathcal{U}$ and is mapped to $(x,y-1)$. If $y-1\ge (n+1)x$, then this new element contradicts the minimality of $u$. Hence we must have $y = (n+1)x$. But then $y-x=nx$, so we can apply Lemma \ref{x=y (mod n)} to find that $nx=y-x=0$. Thus $y=x=0$, which is impossible.

A symmetric argument shows that $x<(n+1)y$ for $(x,y)\neq (1,0)$.
\end{proof}

\subsection{Finiteness} \label{finiteness}

Although it is simple to prove, in $\mathbb{Z}^d$, that Ulam sets are always infinite (see \cite{noah}), the situation is much more complicated in $\mathbb{Z}\times (\mathbb{Z}/n\mathbb{Z})$. In this case, we can indeed have Ulam sets with finitely many elements. A simple example is obtained by taking the initial set $S=\{(1,0),(1,1),\ldots, (1,n-1)\}$, when $n\ge 5$. In this case, it is clear that no other elements outside of the initial set can be added. A more subtle example with only three generators is obtained with $S=\{(1,0),(1,1),(2,5)\}$ in $\mathbb{Z} \times (\mathbb{Z}/8\mathbb{Z})$. In this example, several other elements will also be included, but if we compute enough terms, we can see that the process must terminate eventually and that there is no element with $x$-coordinate greater than $51$ (see Figure \ref{fig:finite-set}).

This example with three generators suggests trying to find one with only two generators. Surprisingly, however, we found no finite Ulam set with two generators despite checking, with a computer, all associated lattices generated by $(-a,b)$ for $a,b<200$ and up to elements with $x=1000$.

\begin{conjecture} \label{conj:infinite}
Every Ulam set in a commutative group with two generators is infinite.
\end{conjecture}

Note that we can restrict our attention to $\mathbb{Z} \times (\mathbb{Z}/n\mathbb{Z})$ since we know that every Ulam set (in a commutative group) with two linearly dependent generators can be embedded in $\mathbb{Z}\times (\mathbb{Z}/n\mathbb{Z})$. Note also that this conjecture does not hold in the case of $\mathcal{V}$-sets; for example, the $\mathcal{V}$-set with associated lattice generated by $(-3,3)$ contains only five elements.

We know, however, that when $\gcd(a,b)=1$, the Ulam set with associated lattice generated by $(-a,b)$ can be embedded in $\mathbb{Z}$, so it must be infinite. The following theorem allows us to improve this to include all cases where $\gcd(a,b)<5$.

\begin{theorem}
\label{5 or more in last column}
Let $\mathcal{U}$ be a finite Ulam set in $\mathbb{Z} \times (\mathbb{Z}/n\mathbb{Z})$ and let $x_{\max}$ be the greatest $x$-coordinate of elements of $\mathcal{U}$. Then $\mathcal{U}$ contains at least $5$ elements of the form $(x_{\max},y)$.
\end{theorem}

\begin{proof}
If $\mathcal{U}$ contains a single element of the form $(x_{\max},y)$, then the sum of this element with an element of $\mathcal{U}$ with the second greatest $x$-coordinate clearly has a unique representation, contradicting the maximality of $x_{\max}$.

If $\mathcal{U}$ contains two or three elements of this form, then the sum of any two of them again has a unique representation and greater $x$-coordinate.

Now suppose that $\mathcal{U}$ contains four elements of this form: $u_1$, $u_2$, $u_3$, and $u_4$. Then the only possible second representation of $u_1+u_2$ is $u_3+u_4$, so $u_1+u_2=u_3+u_4$. Similarly, $u_1+u_3=u_2+u_4$ and $u_1+u_4=u_2+u_3$. Thus, we have the following system of equations:
$$\left\{
  \begin{array}{l}
    u_1+u_2=u_3+u_4\\
    u_1+u_3=u_2+u_4\\
    u_1+u_4=u_2+u_3
  \end{array}
\right.$$
Subtracting the first two equations yields $u_2-u_3=u_3-u_2$, so $2u_2=2u_3$. Similarly we can obtain $2u_i=2u_j$ for all pairs $i,j\in\{1,2,3,4\}$. If $n$ is odd, then $u_i=u_j$ (because then $2$ if invertible modulo $n$), which is impossible. If $n$ is even then it only forces $u_i\in \{u_j,u_j+(0,\frac{n}{2})\}$. This means, however, that for every $i\in\{1,2,3,4\}$, $u_i$ is either equal to $u_1$ or $u_1+(0,\frac{n}{2})$. Since the four $u_i$'s are different and there are only two choices, this is also impossible. Hence, there must be at least $5$ elements of the form $(x_{\max},y)$.
\end{proof}

\begin{corollary} \label{infinite234}
All Ulam sets in $\mathbb{Z}\times (\mathbb{Z}/2\mathbb{Z})$, $\mathbb{Z}\times (\mathbb{Z}/3\mathbb{Z})$ and $\mathbb{Z}\times (\mathbb{Z}/4\mathbb{Z})$ are infinite.
\end{corollary}
\begin{proof}
Since there are less than $5$ possible values of $y$, the condition of Theorem \ref{5 or more in last column} cannot be fulfilled, so the Ulam sets must be infinite.
\end{proof}

\section{Higher-dimensional $\mathcal{V}$-sets} \label{sec:V-sets}

In this section, we study $\mathcal{V}$-sets, the variant of Ulam sets where we don't require the summands in the representations to be distinct. These sets share many properties with Ulam sets. In particular, Conjecture \ref{alpha1/alpha2 conjecture} appears to also hold in the case of $\mathcal{V}$-sets. Moreover, the properties pertaining to associated lattices, discussed at the beginning of Section \ref{ZxZn}, still apply to $\mathcal{V}$-sets.

The case of $\mathcal{V}$-sequences (in $\mathbb{Z}$) have already been studied by Kuca \cite{V-sequences}, so we focus on $\mathcal{V}$-sets in $\mathbb{Z}^2$.

\subsection{The column phenomenon}

In Section \ref{free-col}, we extended the column phenomenon first observed by Kravitz and Steinerberger \cite{noah} to a noncommutative setting. We will now prove a generalization of this phenomenon in commutative settings that will allow us to extend it to $\mathcal{V}$-sets.

\begin{definition} If $S$ is a subset of $\mathbb{Z}_{\ge 0}^2$, we say that the $x$-column is eventually periodic with period $p$ when, for a sufficiently large $y$, $(x,y)\in S$ if and only if $(x,y+p)\in S$.
\end{definition}

Kravitz and Steinerberger proved that if an Ulam set in $\mathbb{Z}^2$ has a single generator lying on the first column ($x=0$), then all of the columns are eventually periodic.

\begin{figure}[htp!]
    \centering
    \includegraphics[width=10cm]{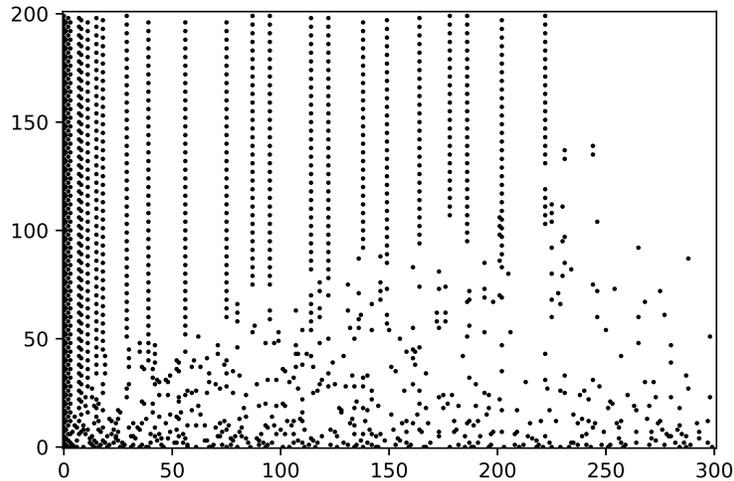}
    \caption{The set $\mathcal{V}(\{(0,1),(1,0),(6,0)\})$. Despite some chaotic behavior near the $x$-axis, regular columns arise for sufficiently large $y$.}
\end{figure}

For $\mathcal{V}$-sets, however, a single generator $(0,a)$ in the first column will generate a full sequence of points in this column: $(0,2a)$, $(0,3a)$, $(0,5a)$, $(0,7a)$, $(0,9a)$, and so on. The periodic behavior in the first column suggests the following generalization:

\begin{theorem} 
\label{col phenomenon}
Let $\mathcal{S}$ either be an Ulam set or a $\mathcal{V}$-set in $\mathbb{Z}^2$ for which the column $x=0$ is eventually periodic. Then all the columns of $\mathcal{S}$ are eventually periodic.
\end{theorem}

\begin{proof} 
We proceed by inducting on $x$. Fix some $x>0$ and suppose that all previous columns are eventually periodic. Let $P$ be the least common multiple of their periods. 

We will need to count the number of representations of $(x,y)$ as sums of previous elements. Moreover, since the exact number of representations does not matter once there is more than one, we will encode the number of representations with a symbol in $\{0,1,2_+\}$ (as in Theorem \ref{thm: free-periodicity}), where the symbols $0$ and $1$ mean there are $0$ and $1$ representations, respectively, and $2_+$ means there are two or more representations. Finally, when we count representations, we can either require the summands to be distinct or not require this restriction, depending on whether we are working with an Ulam set or a $\mathcal{V}$-set. The proof works equally well in both settings.

We first ignore the first column and define $b_y\in \{0,1,2_+\}$ to be the number of representations of $(x,y)$ as a sum of elements from the other columns. We show that $b_y$ is eventually periodic with period $P$.

Indeed, suppose that we have a representation $(x,y)=(x_1,y_1)+(x_2,y_2)$, with $y_1\le y_2$. It is clear that, for sufficiently large $y$, the point $(x_2,y_2)$ must come from the periodic section of its column. If $(x_1,y_1)$ also does, then we have another representation $(x,y)=(x_1,y_1-P)+(x_2,y_2+P)$, so $b_y=2_+$. Similarly $b_{y+P}=2_+$ because $(x,y+P)=(x_1,y_1)+(x_2,y_2+P)$. Furthermore, if $(x,y+P)$ has such a representation, then $(x,y)$ also does and we again have $b_y=b_{y+P}=2_+$.

Now, suppose all representations of both $(x,y)$ and $(x,y+p)$ use an element from the non-periodic transient phase of its column. Then any representation $(x,y)=(x_1,y_1)+(x_2,y_2)$ yields the representation $(x,y+P)=(x_1,y_1)+(x_2,y_2+P)$. Similarly, a representation of $(x, y+P)$ yields one for $(x,y)$. We therefore have a bijection between representations of $(x,y)$ and $(x,y+P)$, which shows that $b_y=b_{y+P}$. 

We now need to split the first column into its finite transient phase and its infinite periodic section. Since its minimal period divides $P$, for any congruence class modulo $P$, either $(0,y)\in \mathcal{S}$ for all sufficiently large $y$ in this equivalence class, or $(0,y)\not\in \mathcal{S}$ for all sufficiently large $y$ in this equivalence class. Let $C$ be the set of congruence classes modulo $P$ that contain the values $y$ for which $(0,y)$ is eventually always in $\mathcal{S}$, and let $T$ be the finite set of elements of $\mathcal{S}$ not in one of these classes.

Let $c_y \in \{0,1,2_+\}$ be the number of representations of $(x,y)$, where we now also take into account the infinite periodic section of the first column.

For a congruence class $R\in \mathbb{Z}/P\mathbb{Z}$, consider the congruence classes $R-S$ for $S\in C$. Then each $y'$ in one of those classes with $(x,y')\in \mathcal{S}$ yields a representation of $(x,y)$ for all sufficiently large $y\in R$. Thus, if there are two or more such elements, $c_y=2_+$ for all sufficiently large $y \in R$. If there is one such element, then $c_y=b_y+1$ for all sufficiently large $y\in R$, and if there is no such element then $c_y=b_y$. Thus, the sequence $(c_y)$ is still eventually periodic with period $P$.

All that's left for us to consider is the effect of $T$, the set of elements in the non-periodic transient phase of the first column. Let $a_y$ be the indicator sequence for the elements of the $x$-column ($a_y=1$ if $(x,y)\in \mathcal{S}$ and $a_y=0$ otherwise). It is clear that the sequence $(a_y)$ is determined by $c_y$ and $T$ recursively as follows:

If $c_y=0$ and there exists a unique $t \in T$ for which $a_{y-t}=1$, then $a_y=1$. If $c_y=1$ and there exists no $t \in T$ for which $a_{y-t}=1$, then $a_y=1$. Otherwise, $a_y=0$.

Let $m$ be the maximal element of $T$. Then, for sufficiently large $y$, $a_y$ is uniquely determined by $a_{y-1}, a_{y-2}, \ldots, a_{y-m}$ and the residue of $y$ modulo $P$. However, since there are only $2^m$ possible combinations of values for $a_{i-1}, a_{i-2}, \ldots, a_{i-m}$, there must eventually be some $y_0$ and $k>0$ such that $a_{y_0-j}=a_{y_0+kP-j}$ for all $1 \le j \le m$. But since $a_{y_0}$ depends solely on $a_{y_0-1}, \ldots a_{y_0-m}$, we must also have $a_{y_0}=a_{y_0-kP}$. Continuing to apply the recursive relation therefore implies that $a_y$ is eventually periodic (when $y \ge y_0$) with period $kP$.
\end{proof}

Note that our proof only provides large bounds on the periods of the columns. Indeed, the period could increase by a factor of up to $2^m$ whenever $x$ increases by $1$. If $T$ contains a single element $t$ (as in the case of $\mathcal{V}$-sets with a single generator on the $y$-axis), then we can actually improve this bound and show that the period at most doubles when $x$ increases by $1$ (consider each congruence class modulo $t$ separately in the third part of the proof). In many cases, however, this doubling rarely occurs, as noted in \cite{noah} for the case of Ulam sets. It would therefore be interesting to try to obtain better bounds on the periods.

\subsection{The $\mathcal{V}$-set with two independent generators}

Contrary to the case of Ulam sets, the $\mathcal{V}$-set on two generators with associated lattice zero does not have a nice simple lattice structure. We can, however, obtain an interesting result about the structure of this $\mathcal{V}$-set, which we will consider to be embedded in $\mathbb{Z}^2$ with initial set $\{(0,1),(1,0)\}$ (since all $\mathcal{V}$-sets with linearly independent generators are structurally equivalent to this one). For simplicity, we will refer to $\mathcal{V}(\{(0,1),(1,0)\})$ simply as $\mathcal{V}_0$.

\begin{figure}[htp!]
    \centering
    \includegraphics[width=10cm]{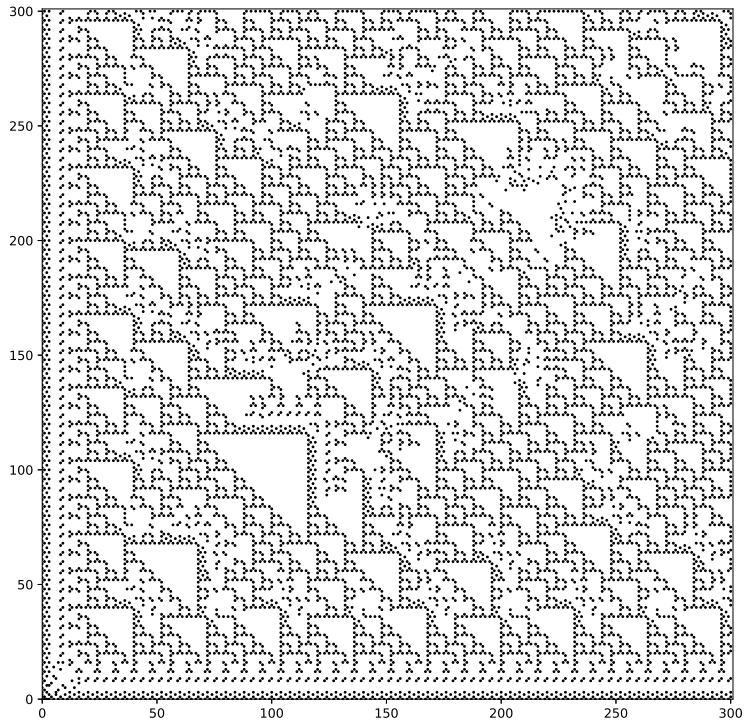}
    \caption{The set $\mathcal{V}_0=\mathcal{V}(\{(0,1),(1,0)\})$.}
\end{figure}

\begin{theorem}
\label{V0}
Let $E$ be the set of elements $(x,y)$ such that $(x,y)\equiv (0,1),(0,3),(1,0),(3,0)$ or $(2,2) \pmod{4}$. Then all the elements of $\mathcal{V}_0$ are in $E$ except $(1,1)$, $(2,0)$, $(0,2)$, $(3,2)$, $(2,3)$, $(6,3)$, $(3,6)$, $(9,6)$, $(6,9)$, $(10,5)$, $(5,10)$, $(14,5)$, and $(5,14)$.
\end{theorem}

\begin{proof}
According to Theorem \ref{col phenomenon}, all the columns are eventually periodic. We will need to explicitly compute the elements of the first $12$ columns of $\mathcal{V}_0$:
\begin{itemize}
    \item $(0,y)\in \mathcal{V}_0 \iff y \equiv 1 \pmod{2}$ or $y=2$
    \item $(1,y)\in \mathcal{V}_0 \iff y \equiv 0 \pmod{4}$ or $y=1$
    \item $(2,y)\in \mathcal{V}_0 \iff y \equiv 2 \pmod{4}$ and $y\ge 6$ or $y\in \{0,3\}$
    \item $(3,y)\in \mathcal{V}_0 \iff y \equiv 0 \pmod{4}$ and $y\ge 12$ or $y\in \{0,2,6\}$
    \item $(4,y)\in \mathcal{V}_0 \iff y \in \{1,7\}$
    \item $(5,y)\in \mathcal{V}_0 \iff y \in \{0,10,14\}$
    \item $(6,y)\in \mathcal{V}_0 \iff y \in \{2,3,9\}$
    \item $(7,y)\in \mathcal{V}_0 \iff y \in \{0,4,16\}$
    \item $(8,y)\in \mathcal{V}_0 \iff y \equiv 3 \pmod{4}$ and $y\ge 19$ or $y\in \{1,13\}$
    \item $(9,y)\in \mathcal{V}_0 \iff y \equiv 0 \pmod{4}$ and $y\ge 16$ or $y\in \{0,6\}$
    \item $(10,y)\in \mathcal{V}_0 \iff y \in \{2,5\}$
    \item $(11,y)\in \mathcal{V}_0 \iff y=0$
\end{itemize}
It is just a matter of computation to check that this does hold. Note that we have symmetric results since everything still holds when we swap $x$ and $y$. Now we split the proof into three parts.

First, we show that $\mathcal{V}_0$ contains no point with both odd $x$ and odd $y$, except $(1,1)$.
If $x$ and $y$ are both odd, we have $(x,y)=(x,0)+(0,y)$. Now write $x=4k_1+r_1$ and $y=4k_2+r_2$, with $r_1,r_2 \in \{1,3\}$. If $k_1\le 2$ or $k_2 \le 2$, then $(x,y)$ is in one of the first $12$ columns or one of the first $12$ rows of $\mathcal{V}_0$. This case has already been dealt with. If $k_1,k_2>3$, however, we have another representation $(x,y)=(r_1,4k_2)+(4k_1,r_2)$. Thus, $(x,y)$ cannot be in $\mathcal{V}_0$.

Second, we now show that $\mathcal{V}_0$ contains no elements with both coordinates even and with at least one of the coordinates divisible by $4$, except $(0,2)$ and $(2,0)$. By symmetry, we can assume without loss of generality that $x$ is divisible by $4$, so $(x,y)=(4\ell,2k)$ for some $\ell$ and $k$. Suppose that $\ell>2$ and $k>1$ (otherwise $(x,y)$ is in one of the first $12$ columns). Then we have two distinct representations $$(x,y)=(4\ell,1)+(0,2k-1)=(4\ell,3)+(0,2k-3).$$ Hence $(x,y)$ is not in $\mathcal{V}_0$.

Third, we show that $\mathcal{V}_0$ contains no point with $x$ odd and $y \equiv 2 \pmod{4}$, except $(3,2)$, $(3,6)$, $(5,10)$, $(5,14)$, and $(9,6)$. Once again, suppose that $x,y \ge 12$. Then we have two distinct representations $$(x,y)=(x-2,0)+(2,y)=(x-d,2)+(d,y-2),$$ where $d=1$ if $x\equiv 3 \pmod{4}$ and $d=3$ if $x\equiv 1 \pmod{4}$. Hence, $(x,y)$ is not in $\mathcal{V}_0$. By symmetry, we also have that $\mathcal{V}_0$ contains no element with odd $y$ and $x \equiv 2 \pmod{4}$ except for $(2,3)$, $(6,3)$, $(10,5)$, $(14,5)$, $(6,9)$. 

Combining these three results, we see that the only combinations of remainders modulo $4$ for $x$ and $y$ that have not been excluded are $(0,1),(0,3),(1,0),(3,0)$ and $(2,2)$, so all the elements of $\mathcal{V}_0$ except for the few listed exceptions are in one of these classes.
\end{proof}

This is very similar to the condition in Theorem \ref{E condition}: it is indeed easy to check that the sum of two elements in $E$ is never in $E$. Our result therefore implies that there is some finite set $T\subset \mathcal{V}_0$ such that the representation of sufficiently large elements in $\mathcal{V}_0$ necessarily uses a summand from $T$.

\begin{corollary}
\label{V0 corollary}
Let $$T=\{(1,1),(2,0),(0,2),(3,2),(2,3),(6,3),(3,6),(9,6),(6,9),(10,5),(5,10),(14,5),(5,14)\}.$$
Then every element of $\mathcal{V}(\{(0,1),(1,0)\})$ outside of $T$ (and the generators) must use an element of $T$ in its (unique) representation as a sum of two previous elements.
\end{corollary}

\begin{proof}
This follows directly from Theorem \ref{V0} and the fact that the sum of two elements in $E$ is never in $E$.
\end{proof}

Even though this is enough to imply regularity in one dimension (Theorem \ref{E condition}), it is unfortunately not necessarily the case in higher dimensions, and the structure of $\mathcal{V}_0$ still appears to be quite hard to describe. Theorem \ref{col phenomenon} implies that the columns and rows are eventually periodic, but the transient phases seem to be too long for a lattice structure to emerge.

We believe that Corollary \ref{V0 corollary} could help prove that $\mathcal{V}_0$ has positive asymptotic density, which would be an interesting result for a $\mathcal{V}$-set with no lattice structure. It has also allowed us to efficiently compute, with a computer, all the elements of $\mathcal{V}_0$ with $x$ and $y$ up to $50000$. These computations showed that the density empirically seems to be approximately $0.1218$ (density of $0.05908$ for the points of type $(0,1)/(1,0)$, $0.05959$ for those of type $(0,3)/(3,0)$ and $0.00314$ for those of type $(2,2)$). Note that the points of type $(2,2)$ are much rarer than the other types.

\section{Conclusion and open problems} \label{conclusion}

We conclude by gathering a few open questions that arose during the present investigation.

\subsection{Complete characterization of elements in $\mathcal{U}(\{0,1\})$} We fully characterized all terms in $\mathcal{U}(\{0,1\})$ with exactly one $1$ (and, by Theorem \ref{thm:symmetry}, those with exactly one $0$). We also investigated symmetries (reverses, bit-wise complements, and palindromes) and general conditions for words with exactly two $1$'s. However, the full characterization of all words in $\mathcal{U}(\{0,1\})$ still remains unsolved. In Theorem \ref{thm:one1}, we found that the condition for a word with one $1$ to be in the Ulam set is a modular restriction of a binomial coefficient. We predict that a similar modular restriction must suffice for a binomial coefficient or sum of binomial coefficients that correspond to a word with more than one $1$. 

In particular, finding the exact number of words of length $n$ in $\mathcal{U}(\{0,1\})$ and the asymptotic density of the Ulam set (Conjecture \ref{conj:density}) remains an interesting problem for future research.

\subsection{Ulam sets in matrix groups} The idea of Ulam sets arising from non-abelian settings can be extended to matrices, where we use the determinant of a matrix as our notion of size. To ensure that this notion is suitable for generating an Ulam set, our starting matrices must have determinants greater than $1$. The generating matrices must also not commute.

If the generating matrices are such that every matrix representable as a product of the starting matrices must have a unique representation, then any matrix can be represented uniquely as a word on the alphabet containing the starting matrices. Hence, the Ulam set on two generators in the matrix setting would be isomorphic to $\mathcal{U}(\{0,1\}).$ Therefore, Ulam sets in matrix groups extend the study of Ulam sets in free groups. Matrices allow us to add new conditions on our set; we focus on the case where there is a relation between the generating matrices that would allow for a matrix to have a non-unique representation as the product of the starting generators. In particular, the Ulam set $\mathcal{U}(\{A, B\})$, where
\begin{center}
    $A=\begin{pmatrix}
    0 & 4 \\
    -4 & 4
    \end{pmatrix}
    \quad \text{and} \quad 
    B=\begin{pmatrix}
    0 & 8 \\
    -8 & 0
    \end{pmatrix}$
\end{center}
satisfy $A^3=B^2$, is an interesting area for future investigation.

\subsection{Decomposition into sums or products of the generators} We believe that the function $\alpha$ keeping track of the contribution of each generator to an element could be very important for a better understanding of Ulam sets.
In particular, in commutative settings, further work on Conjecture \ref{alpha1/alpha2 conjecture} stating that the ratio $\frac{\alpha_1}{\alpha_2}$ stabilizes might be interesting. A first step toward this conjecture would be to improve the bounds on $\frac{\alpha_1}{\alpha_2}$ given by Theorems \ref{alpha1/alpha2} and \ref{alpha1/alpha2 a=b}.

Note that we can define a similar $\alpha$ function in noncommutative settings. The ratio does not necessarily seem to stabilize in this case, but we believe that it is still possible to bound it in some cases. In particular, for Ulam sets of matrices where the only relation between the two generators $A$ and $B$ is $A^3=B^2$, we believe that $\frac{\alpha_A}{\alpha_B}\ge 1$ for all elements except the generator $B$ itself. There also seems to be an infinite class of matrices in this Ulam set satisfying $\frac{\alpha_A}{\alpha_B}= 1$. Further research on the noncommutative version of this phenomenon could also be enlightening.

\subsection{Regularity conjecture} We conjectured that, for even $a\ge 2$ and sufficiently large $b$ (depending on $a$), the Ulam set with associated lattice generated by $(-a,b)$ is always regular. We settled the case $a=2$ with Theorems \ref{regularity a=2} and \ref{b=2^e}, and the case with $a=4$ and $b\equiv 1 \pmod{4}$ has already been solved in \cite{regularity 4 v}. We believe that similar work could be applied for other small cases, but new methods will likely be necessary for the general case. Since proofs of regularity appear to often require a lot of casework, we think that computer-assisted proofs may be helpful for further advances.

Note also that we focused mostly on Ulam sets but that we observed regularity for many $\mathcal{V}$-sets in $\mathbb{Z}\times(\mathbb{Z}/n\mathbb{Z})$ as well. A deeper study of these cases could bring a better understanding of regularity phenomenons.

\subsection{Better conditions for finiteness} Another interesting path for further research is the characterization of finite Ulam sets based on their sets of generators. We conjectured that Ulam sets with two generators (in commutative groups) cannot be finite. Any new partial result on this conjecture would be quite interesting. Moreover, we studied finiteness only in commutative groups, but this could also be studied in the noncommutative case. 

\subsection{Density of each row in $\mathbb{Z}\times(\mathbb{Z}/n\mathbb{Z})$} We observed that, in general, the elements of Ulam sets in $\mathbb{Z}\times(\mathbb{Z}/n\mathbb{Z})$ are not equally distributed between the different rows (values of $y$). For example, if we take $n=3$ and initial set $\{(1,1),(1,2)\}$, then there are very few points in the Ulam set with $y$-coordinate $0$, compared to $y=1$ or $y=2$. Studying this phenomenon deeper could be a promising avenue for future research.

\begin{figure}[htp!]
    \centering
    \includegraphics[width=14cm]{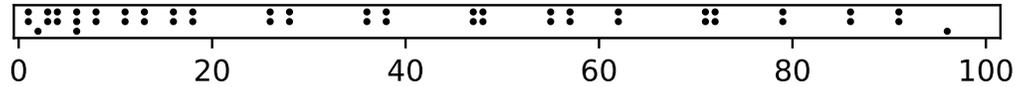}
    \caption{The Ulam set generated by $S=\{(1,1),(1,2)\}$ in $\mathbb{Z}\times (\mathbb{Z}/3\mathbb{Z})$.}
\end{figure}

\section*{Acknowledgements}
This research was conducted under the auspices of Noah Kravitz's summer research program. We wish to thank him for mentoring us throughout this project and for helping us with the writing of this paper. We also thank Borys Kuca for pointing out the connection between our Theorem \ref{E condition} and Theorem 6.3.2 from \cite{ross-thesis}.  This paper benefited from the suggestions of an anonymous referee.

\end{document}